\documentclass[stsy,sglanonrev]{informs4}

\usepackage{natbib}
 \NatBibNumeric
 \def\bibfont{\small}%
 \bibpunct[, ]{[}{]}{,}{n}{}{,}%

\usepackage[colorlinks=true,breaklinks=true,bookmarks=true,urlcolor=blue,citecolor=blue,linkcolor=blue,bookmarksopen=false,draft=false]{hyperref}

\TheoremsNumberedThrough     

\EquationsNumberedThrough    

\usepackage{amsmath,amsfonts,amssymb,amsbsy,mathrsfs,bbm}

\usepackage{mathtools}

\usepackage{algorithm}
\usepackage{algpseudocode}

\usepackage{enumitem}

\usepackage{caption}
\usepackage[labelfont=sf]{subcaption}
\usepackage{tikz}
\usepackage{color-edits}
\addauthor{rp}{blue}

\newcommand{\death}{{\beta}}

\renewcommand{\theARTICLETOP}{}
\usepackage{fancyhdr}
\begin{document}




\TITLE{Frank-Wolfe Recursions for the Emergency Response Problem on Measure Spaces}

\ARTICLEAUTHORS{%
\AUTHOR{Di Yu}
\AFF{Department of Statistics,
Purdue University}

\AUTHOR{Shane G. Henderson}
\AFF{School of Operations Research and Information Engineering, Cornell University}

\AUTHOR{Raghu Pasupathy}
\AFF{Department of Statistics,
Purdue University}
} 

\ABSTRACT{%
We consider an optimization problem over measures for emergency response to out-of-hospital cardiac arrest (OHCA), where the goal is to allocate volunteer resources across a spatial region to minimize the probability of death. The problem is infinite-dimensional and poses challenges for analysis and computation. We first establish structural properties, including convexity of the objective functional, compactness of the feasible set, and existence of optimal solutions. We also derive the influence function, which serves as the first-order variational object in our optimization framework. We then adapt and analyze a fully-corrective Frank-Wolfe (fc-FW) algorithm that operates directly on the infinite-dimensional problem without discretization or parametric approximation. We show a form of convergence even when subproblems are not solved to global optimality. Our full implementation of fc-FW demonstrates complex solution structure even in simple discrete cases, reveals nontrivial volunteer allocations in continuous cases, and scales to realistic urban scenarios using OHCA data from the city of Auckland, New Zealand. Finally, we show that when volunteer travel is modeled through the $L_1$ norm, the influence function is piecewise strictly concave, enabling fast computation via support reduction. The proposed framework and analysis extend naturally to a broad class of $P$-means problems.
}%



\maketitle

%

\section{INTRODUCTION.}\label{sec:intro} 
Out-of-hospital cardiac arrest (OHCA) is a leading cause of death worldwide, with survival outcomes highly dependent on rapid intervention. The probability of survival decreases by approximately 7–10\% for each minute without treatment~\cite{baekgaard2017effects}. Early response, particularly cardiopulmonary resuscitation (CPR) administered by trained individuals within the first few minutes, can significantly increase the chance of survival~\cite{nichol1999cumulative, sasson2010predictors,yan2020global}.

To reduce delays in emergency response, community first responder (CFR) systems have been specifically designed for OHCA patients and are now widely adopted in many countries. These systems supplement traditional emergency medical services (EMS) by dispatching trained volunteers via smartphone applications. In Europe, 19 of 29 countries have implemented such systems specifically for OHCA \cite{oving2019first}. Mobile apps like PulsePoint in the United States \cite{pulsepoint2020} and GoodSAM in New Zealand and the UK alert nearby volunteers based on GPS location \cite{goodsam2020}. Volunteers are often able to reach patients before ambulances, especially in congested urban areas where EMS response may be delayed. CFR systems thus depend crucially on volunteers, and in particular, on where volunteers are located over a geographical region of interest. 

CFR systems do not have direct control over where volunteers are located. Still, the problem of determining the optimal probability distribution for volunteer locations is important \cite{van2024modeling}. First, it is important because the optimal objective function value provides a lower bound on the probability of death for a given number of volunteers, which is important for determining the number of volunteers needed in a city to materially impact survival rates from OHCA. Second, the optimal solution can provide insight into where to focus recruitment efforts for additional volunteers.

However, the question of where to locate volunteers so as to minimize a performance criterion, such as the expected response time to OHCA events, or the probability of death resulting from the next arriving OHCA event, is challenging~\cite{van2024modeling,waalewijn2001survival}. This question is investigated in great detail in~\cite{van2024modeling}, where the authors develop a stochastic model for CFR response times in Auckland, New Zealand, and then examine which spatial distribution of volunteers maximizes the expected survival rate under randomly occurring OHCA events. To pave the way for implementable solutions, two key modeling assumptions are made. First, the volunteers are distributed over the region in question according to a Poisson point process. This assumption allows the computation of the response-time distribution which, when combined with the knowledge of patient survival rates as a function of response time, yields survival rates, and enables the development of an oracle for the objective function. Second, since the decision variable is the intensity measure corresponding to the volunteer Poisson point process, the resulting feasible region is the infinite-dimensional space of functions satisfying the classical postulates~\cite{1992res} of a spatial Poisson intensity measure. To circumvent the need to handle the resulting infinite-dimensional optimization problem, the authors in~\cite{van2024modeling} divide the geographical region into a finite number of \emph{area units}, and restrict attention only to intensity measures that are uniform over each area unit. The resulting finite-dimensional convex optimization problem is then solved by formulating the Karush-Kuhn-Tucker conditions, and then employing a greedy algorithm for solution.   

In this paper, our central interest is the \emph{infinite-dimensional OHCA emergency response problem}. In particular, we ask whether the OHCA emergency response problem considered in~\cite{van2024modeling} can be solved without a priori finite-dimensionalization, that is, can we devise an implementable algorithm that solves the infinite-dimensional problem as is? This question is challenging because it entails devising an algorithm capable of efficiently searching through the space of compactly supported Poisson intensity measures, while guaranteeing approach to an optimum in some rigorous sense. 

Solving the infinite-dimensional OHCA emergency response problem is of theoretical interest, but also motivated by several other considerations. First, a priori finite-dimensionalization as in~\cite{van2024modeling} is a computational convenience that yields solutions that are potentially sub-optimal with respect to the true infinite-dimensional problem. Algorithms that solve the infinite-dimensional problem shed light on the nature and extent of such sub-optimality. Second, and as we shall observe in further detail, the OHCA emergency response problem is a close cousin of a popular and large class of problems called P-means~\cite{2002molzuy,2000oka}, sometimes described as a randomized variant of the $k$-means clustering problem. Our investigations into the infinite-dimensional OHCA emergency reponse problem are thus relevant for efficiently solving the P-means problem. Third, and perhaps most important, our results shed light on the kinds of infinite-dimensional measure optimization problems that are computationally tractable. We shall see that we are able to efficiently solve many OHCA emergency response problems. It seems that such efficiency is afforded primarily by the ability to explicitly calculate the infinite-dimensional directional derivative (called the \emph{von Mises derivative}) of the objective function for embedding into the fully-corrective Frank-Wolfe algorithm. This embedding allows us to (weakly) approach an optimal measure through a simple addition and updating of ``particles.'' In a nutshell, the existence and good behavior of the von Mises derivative allows a type of ``dimension reduction,'' analogous to that observed by Kiefer~\cite{1974kie} in the context of experimental design.


\subsection{Problem Setup}
To formally state the problem of optimally allocating volunteers for OHCA response, let $\mathcal{S} \subset \mathbb{R}^2$ be a compact set and let $\eta \in \mathcal{P}(\mathcal{S})$ denote a probability distribution supported on $\mathcal{S}$ modeling the random location of an OHCA. Let $\beta: [0,\infty) \to [0,1]$ be a \emph{continuous, increasing, strictly concave} function modeling the \emph{probability of death} as a function of response time. Suppose that the availability of volunteers follows a spatial Poisson point process over $\mathcal{S}$ with mean measure $\mu$, where $\mu$ is a non-negative, finite Borel measure on \( \mathcal{S} \) with total mass \( \mu(\mathcal{S}) = b \). We refer to $\mu$ as the \emph{volunteer measure}.  (See~\cite{res1987} for a treatment of spatial Poisson processes.) Let \( R_{x,\mu} \in \mathbb{R}^+ \) denote the random response time to an incident occurring at location \( x \in \mathcal{S} \) under the volunteer measure \( \mu \). The probability of death for the next OHCA incident is then
\begin{equation}\label{obj}
J(\mu) = \int_{\mathcal{S}} \eta(\mbox{d}x)\int_0^{\infty} P(\death(R_{x,\mu}) > u) \, \mbox{d}u = \int_{\mathcal{S}} \eta(\mbox{d}x)\int_0^{1} P(\death(R_{x,\mu}) > u) \, \mbox{d}u. \end{equation}
Assuming volunteers travel at constant speed \( v \), the probability that \( R_{x,\mu} > t \) equals the probability that no volunteer is within the ball \( B(x,vt) := \{ y \in \mathbb{R}^2 : \|y - x\| \leq vt \} \). Without loss of generality, we set \( v = 1 \), and let \( \death^{-1}(u) := \inf\{ t : \death(t) \geq u \} \) denote the inverse of the death probability function. Then, using the void probability formula for a Poisson point process with intensity measure \( \mu \),
\begin{align}\label{objexp}
J(\mu)
&= \int_{\mathcal{S}} \eta(\mathrm{d}x) \int_0^1 P(R_{x,\mu} > \death^{-1}(u)) \, \mathrm{d}u \notag \\
&= \int_{\mathcal{S}} \eta(\mathrm{d}x) \int_0^1 \exp\left\{-\mu\left(B(x, \death^{-1}(u)) \cap \mathcal{S}\right)\right\} \, \mathrm{d}u \notag \\
&= \int_{\mathcal{S}} \eta(\mathrm{d}x) \int_0^{\infty} \exp\left\{-\mu\left(B(x, t) \cap \mathcal{S}\right)\right\} \, \mathrm{d}\death(t).
\end{align} We thus seek the optimal volunteer measure $\mu$ that minimizes the  probability of death, represented by the functional \( J(\mu) \), leading to the optimization problem
\begin{alignat}{2}
\begin{cases}
\min\limits_{\mu} & \quad J(\mu) \\
\text{subject to} & \quad \mu \in \mathcal{M}^+(\mathbb{R}^2, b), \tag{$P_0$}
\end{cases}
\label{mainopt}
\end{alignat}
where \( \mathcal{M}^+(\mathbb{R}^2, b) \) denotes the set of non-negative Borel measures supported on \( \mathbb{R}^2 \) with total mass \( \mu(\mathbb{R}^2) = b \), and \( J : \mathcal{M}^+(\mathbb{R}^2, b) \to \mathbb{R} \) is the objective functional defined in~\eqref{objexp}. 

\subsection{Summary and Contribution}

\begin{enumerate}[label=(\alph*)]
    \item We study the structural properties of the (infinite-dimensional) OHCA emergency response problem and propose an equivalent restricted formulation~\eqref{restrictedopt} over the compact convex hull of the incident region, denoted $\mbox{cvx}(\mathcal{S}) \subset \mathbb{R}^2$ (Theorem~\ref{thm:restrict}). We show that the objective functional $J$ is convex, the feasible set $\mathcal{M}^+(\mbox{cvx}(\mathcal{S}), b)$ is compact under the weak$^*$ topology, and optimal solutions exist (Theorems~\ref{thm:feasibleregion}--\ref{thm:existence}). We also derive a closed-form expression for the influence function and establish the existence of the von Mises derivative (Theorem~\ref{thm:vonMises}), defined in Definition~\ref{def:vonmises}, which serves as the first-order variational object in our optimization framework.
    \item We develop theoretical guarantees for the fully-corrective Frank-Wolfe (fc-FW) method in Algorithm~\ref{alg:fcFW}, applied directly to the infinite-dimensional optimization problem over measures. Unlike classical Frank-Wolfe variants that require solving the subproblem of minimizing the influence function to global optimality, our analysis in Section~\ref{sec:fw_methods} establishes a sufficient decrease property (Lemma~\ref{lemma:decrease}) and a complexity bound (Theorem~\ref{thm:convfcFW}) under a mild condition: the influence function value at the selected point is non-positive. This is significant because the influence function is generally nonconvex, making global minimization intractable. We show that fc-FW remains provably convergent (in a certain sense) even when the subproblem is solved only to (near) local optimality. The influence function value at the selected point serves as a practical proxy for the optimality gap and decays at a rate of $\mathcal{O}(1/\sqrt{k})$ (Theorem~\ref{thm:convfcFW}) in the iteration count, $k$. If the subproblem is solved to global optimality, then global convergence is assured, and the rate improves to $\mathcal{O}(1/k)$, matching the standard Frank-Wolfe rate. These properties enable scalable implementation of fc-FW in practical settings.
    \item We provide a full implementation of the fc-FW algorithm on the OHCA emergency response problem and evaluate it across a range of scenarios. In stylized discrete cases with only three or four fixed incident locations, we find that the optimal volunteer measures exhibit complex structure, and that fc-FW computes solutions satisfying a global optimality condition. In continuous settings with uniform and mixture-of-uniform incident distributions, the results show that the optimal volunteer distribution does not simply mirror the incident distribution. We also apply fc-FW to a large-scale, realistic setting using OHCA data from the city of Auckland, New Zealand, demonstrating its scalability and effectiveness at practical scales in complex spatial domains. Furthermore, since the objective function in $P$-means~\cite{2002molzuy,2000oka} shares the same structural form, the proposed method and analysis extend naturally to a broad class of $P$-means problems.
    \item We explore the impact of norm choice for modeling volunteer travel on the objective function. In particular, we compare the standard Euclidean ($L_2$) norm with the Manhattan ($L_1$) norm, which better reflects travel patterns in grid-like urban environments. Under the $L_1$ norm, we show that the influence function is piecewise strictly concave for discrete incident distributions (Lemma~\ref{lemma:piecewiseConcave}), which implies that the support of the optimal measure lies within a structured, finite subset of the domain (Theorem~\ref{thm:supportSubset}). These results offer both theoretical and practical insights for urban deployments and, potentially, enable faster computation.
\end{enumerate}





\subsection{Paper Organization} The rest of this paper is organized as follows. Section~\ref{sec:Literature} reviews related work on optimization over measures. Section~\ref{sec:preliminaries} presents definitions and the Frank-Wolfe framework for measure optimization. Section~\ref{sec:regularity} establishes regularity properties of the problem. Section~\ref{sec:specialCases} provides analytical insights from special cases with discrete incident locations. Section~\ref{sec:implementation} describes the full implementation of the proposed fc-FW algorithm and presents numerical results on both synthetic and real-world datasets. Section~\ref{sec:normMatter} explores the impact of norm choice on the objective. Section~\ref{sec:concluding} concludes and outlines directions for future research.

\section{Literature on Optimization over Measures.}\label{sec:Literature}
Optimization over the space of measures arises in nonparametric statistics~\cite{1960kie,yan2024learning}, signal processing~\cite{eftekhari2019sparse,2017boygeorec}, and machine learning~\cite{2019chublagly,2018mei}. In our setting, the problem~\eqref{restrictedopt} belongs to this infinite-dimensional class. A traditional approach is to discretize the domain—corresponding to the city $\mathcal{S}$ in our emergency response setting—into a finite grid and optimize over weights assigned to each grid point. While simple to implement, this ``grid and optimize'' strategy has significant drawbacks~\cite{2017boygeorec,yu2024det}. It ignores structural properties of the objective \( J \) (such as smoothness or convexity) and incurs rapidly growing computational costs as the grid becomes finer. These limitations hinder both scalability and solution accuracy.

Another widely used strategy for optimizing over measures is \emph{particle methods}~\cite{nitanda2017stoch,chizat2018global,2022bchi,yan2024learning}, which approximate a general measure by an empirical measure constructed as the finite weighted sum of Dirac masses. This reduces the infinite-dimensional optimization problem to a finite-dimensional one. Unlike gridding, particle methods update both the support (particle locations) and the associated weights using various approaches, e.g., Wasserstein-Fisher-Rao gradient flows~\cite{yan2024learning}. 

While particle methods are powerful, they are not necessarily well suited for the emergency response setting considered here. Typically, particle methods fix the number of particles (support points) in advance. However, in the OHCA volunteer allocation problem considered in this paper, the optimal measure can exhibit a complex, nontrivial structure and may not admit a finite-support representation even in simple cases. This makes it difficult to determine in advance how many support points are needed for a faithful approximation. One could certainly successively increase the number of points in a series of stages, but we prefer an alternative approach.

A third strategy for optimization over measures is to use the \emph{Frank-Wolfe} (FW) (also known as conditional gradient) iteration~\cite{2017boygeorec,denoyelle2019sliding,kent2021frankwolfe,yu2024det}. Unlike the previous methods, FW operates directly in infinite-dimensional spaces by iteratively accumulating point masses at strategic locations, without requiring a priori discretization or fixing the number of support points. Recent work~\cite{yu2024det} develops a deterministic Frank-Wolfe (dFW) recursion tailored to probability spaces, using the \emph{influence function} (see Definition~\ref{def:vonmises}) as the first-order variational object. At each iteration, dFW updates the current measure by forming a convex combination with a Dirac measure located at the minimizer of the influence function.

An approach such as dFW is especially attractive for the OHCA emergency response problem because it allows the support of the solution to grow adaptively and captures complex, nontrivial structures in the optimal volunteer measure without requiring assumptions of sparsity or finite support. However, solving the subproblem of finding a global minimizer of the influence function can be difficult in practice, due to the nonconvexity of the influence function. To address this, we establish new convergence guarantees for the fully-corrective variant (fc-FW) (Algorithm~\ref{alg:fcFW}) that do not require solving the subproblem to global optimality (see Section~\ref{sec:fw_methods}). This flexibility makes fc-FW especially well suited for solving the OHCA volunteer allocation problem.

\section{PRELIMINARIES.}\label{sec:preliminaries}
Section~\ref{sec:def} defines key concepts from measure theory, including the space of finite Borel measures, support, and the influence function as the first-order object in measure optimization. Section~\ref{sec:fw_methods} introduces Frank-Wolfe (FW) methods adapted to the infinite-dimensional setting. We review convergence results for the deterministic FW (dFW) algorithm from~\cite{yu2024det} and establish new convergence guarantees for its fully-corrective variant (fc-FW), which serves as the computational framework for our OHCA emergency response problem.

\subsection{Definitions.}\label{sec:def} Suppose $\mathcal{X}$ is a compact convex subset of $\mathbb{R}^2$, and let $(\mathcal{X},\mathcal{B}(\mathcal{X}))$ denote a measurable space, where $\mathcal{B}(\mathcal{X})$ represents the Borel $\sigma$-algebra on $\mathcal{X}$. 
 Let $\mathcal{M}^+(\mathcal{X})$ denote the \emph{space of $\sigma$-finite measures} on this measurable space.
\begin{definition}[Subset of Measure Spaces] In the context of the OHCA emergency response problem, we restrict our focus to measures $\mu$ with a finite total measure $b$, that is, to the set $\mathcal{M}^+(\mathcal{X},b) := \{\mu \in \mathcal{M}^+(\mathcal{X})\, : \, \mu(\mathcal{X})=b\}$.
It is evident that the set $\mathcal{M}^+(\mathcal{X},b)$ is convex, and that  when $b=1$, $\mathcal{M}^+(\mathcal{X},1)$ is the set of probability measures on $(\mathcal{X},\mathcal{B}(\mathcal{X}))$.
\end{definition}
\begin{definition}[Support] The \emph{support} of a measure $\mu \in \mathcal{M}^+(\mathcal{X})$ is \begin{equation}\label{suppdef} \mathrm{supp}(\mu) := \{ x \in \mathcal{X} : \mu(B_\epsilon(x)) > 0\ \text{for all } \epsilon > 0 \},
\end{equation} where \( B_\epsilon(x) \) denotes the open ball centered at \( x \) with radius \( \epsilon \). \end{definition}
\begin{definition}[Influence function and von Mises Derivative]\label{def:vonmises} Suppose $J: \mathcal{M}^+(\mathcal{X},b) \to \mathbb{R}$ is a real-valued function.

The \emph{influence function} $h_{\mu}: \mathcal{X} \to \mathbb{R}$ of $J$ at $\mu \in \mathcal{M}^+(\mathcal{X},b)$ is defined as \begin{equation}\label{influence} h_{\mu}(x) = \lim_{t \to 0^+}\frac{1}{t} \bigg\{ J(\mu + t(b\,\delta_x - \mu)) - J(\mu) \bigg\}, \end{equation} where $\delta_x := \mathbb{I}_{A}(x), A \subset \mathcal{X}$ is the Dirac measure (or atomic mass) concentrated at $x \in \mathcal{X}$~\cite{1983fer,2011fer}. The influence function should be loosely understood as the rate of change in the objective $J$ at $\mu$, due to a perturbation of $\mu$ by a point mass $b\,\delta_x$.

The \emph{von Mises derivative} is defined as
\begin{equation}\label{vonmises}
    J'_{\mu}(\nu-\mu) := \lim_{t \to 0^+}\, \frac{1}{t} \bigg\{J(\mu + t(\nu - \mu)) - J(\mu)\bigg\}, \quad \mu,\,\nu\in \mathcal{M}^+(\mathcal{X},b),
\end{equation}
provided $J'_{\mu}(\cdot)$ is \emph{linear} in its argument, that is, there exists a function $\phi_{\mu}: \mathcal{X} \to \mathbb{R}$ such that
\begin{align} \label{vonlinear}
J'_{\mu}(\nu-\mu) &= \int \phi_{\mu}(x)\, \mathrm{d}(\nu - \mu)(x).
\end{align}
From \eqref{influence} and \eqref{vonmises}, we see that $J'_{\mu}(b\delta_x - \mu) = h_\mu$; and if~\eqref{vonlinear} holds, then $J'_{\mu}(b\delta_x - \mu) = b\phi_{\mu}(x) + \mbox{constant} = h_{\mu}(x)$. As implied by Theorem~\ref{thm:vonMises} in Section~\ref{sec:regularity}, the influence function and von Mises derivative exist for the objective function $J$ in~\eqref{objexp}.  \end{definition}

\begin{definition}[{\em L}-Smooth] The functional $J: \mathcal{M}^+(\mathcal{X},b) \to \mathbb{R}$ is $L$-smooth if it satisfies \begin{align}\label{def:Lsmooth}
    \sup_{x \in \mathcal{X}} |h_{\mu_1}(x)-h_{\mu_2}(x) | \, & \leq \, L \|\mu_1 - \mu_2 \|, \quad \forall \mu_1 ,\, \mu_2 \in \mathcal{M}^+(\mathcal{X},b),
\end{align} where $h_{\mu_1}$ and $h_{\mu_2}$ are corresponding influence functions, and the total variation distance between $\mu_1$ and $\mu_2$ in $\mathcal{M}^+(\mathcal{X},b)$ is defined as
\begin{align}\label{totvar}
    \|\mu_1 - \mu_2 \| &:=\sup_{A\in \mathcal{B}(\mathcal{X})} |\mu_1(A)-\mu_2(A) |,
\end{align} where $\mathcal{B}(\mathcal{X})$ is the Borel $\sigma$-algebra.  As written, the symbol $\| \cdot \|$ appearing in~\eqref{totvar} does not refer to a norm but our use of such notation is for convenience and should cause no confusion.
\end{definition} 

\begin{definition}[Projection in $\mathbb{R}^2$]\label{def:projection} Let $\mathcal{X} \subseteq \mathbb{R}^2$ be a closed and convex set. The \emph{projection operator} $\Pi_{\mathcal{X}}: \mathbb{R}^2 \to \mathcal{X}$ is defined as
\begin{equation}
    \Pi_{\mathcal{X}}(x) := \argmin_{z \in \mathcal{X}} \|x - z\|,
\end{equation}
which maps each point $x \in \mathbb{R}^2$ to its closest point in $\mathcal{X}$ with respect to the Euclidean norm.
For any subset $A \subseteq \mathcal{X}$, the \emph{preimage} of $A$ under $\Pi_{\mathcal{X}}$ is defined as
\begin{equation}
    \Pi^{-1}_{\mathcal{X}}(A) := \{ x \in \mathbb{R}^2 \mid \Pi_{\mathcal{X}}(x) \in A \}.
\end{equation}
This set consists of all points in $\mathbb{R}^2$ whose projection under $\Pi_{\mathcal{X}}$ lies within $A$.
\end{definition}

\subsection{Frank-Wolfe Methods for Measure Optimization}\label{sec:fw_methods}

In this section, we introduce Frank-Wolfe methods for optimizing measures over $\mathcal{M}^+(\mathcal{X},b)$. To motivate our approach, recall the Frank-Wolfe recursion~\cite{1978dunhar} (also known as the \emph{conditional gradient} method~\cite{2015bub}) in finite-dimensional Euclidean spaces. When minimizing a smooth function \( f: \mathbb{R}^d \to \mathbb{R} \) over a compact convex set \( Z \subset \mathbb{R}^d \), the Frank-Wolfe recursion is given by
\begin{equation} \label{eq:FW}
    y_{k+1} = (1 - \eta_k)y_k + \eta_k s_k, \quad s_k := \argmin_{s \in Z} \nabla f(y_k)^\top s,
\end{equation}
where \( \eta_k \in (0,1] \) is a step size. This method is ``projection-free'' and primal in that the iterates \( \{y_k\} \) are always feasible; whether the method is efficient depends on how easily the subproblem can be solved.

To extend Frank-Wolfe to $\mathcal{M}^+(\mathcal{X},b)$, consider a smooth functional \(J : \mathcal{M}^+(\mathcal{X},b) \to \mathbb{R}\)  and write the first-order approximation using the von Mises derivative \( J'_\mu \):
\[
J(u) \approx J(\mu) + J'_\mu(u - \mu), \quad \text{for } u \in \mathcal{M}^+(\mathcal{X},b).
\]
This suggests the following recursion for measures:
\[
\mu_{k+1} = (1 - \eta_k)\mu_k + \eta_k u_k, \quad u_k := \argmin_{u \in \mathcal{M}^+(\mathcal{X},b)} J'_{\mu_k}(u - \mu_k).
\]
As shown in Lemma~5~\cite{yu2024det}, the minimizer of \(J'_\mu(u - \mu)\) is a scaled Dirac measure $b\delta_{x^*(\mu_k)}$, where \(x^*(\mu)\) minimizes the influence function \(h_\mu(x)\) over $\mathcal{X}$.
Thus, the recursion simplifies to:
\begin{equation}\label{dFW}
\mu_{k+1} = (1 - \eta_k)\mu_k + \eta_k b\delta_{x^*(\mu_k)}, \quad x^*(\mu_k) \in \argmin_{x \in \mathcal{X}} h_{\mu_k}(x). \tag{dFW}
\end{equation}
The approach implied by~\eqref{dFW} thus solves an infinite-dimensional optimization problem by iteratively accumulating point masses at optimally chosen locations in \(\mathcal{X}\). Unlike traditional finite-dimensionalization techniques such as gridding, the~\eqref{dFW} recursion dynamically updates its support set without requiring a predefined discretization of the search space. Under standard assumptions of convexity and \(L\)-smoothness of \(J\), the~\eqref{dFW} method guarantees an \(O(k^{-1})\) convergence rate in objective value, as established in Theorem~\ref{thm:dFWcomp}. This result follows from the complexity analysis in~\cite{yu2024det}. 

\begin{theorem}[Complexity~\cite{yu2024det}]\label{thm:dFWcomp}
Suppose $J$ is convex and $L$-smooth, and the step-sizes $\{\eta_k, k \geq 0\}$ in~\eqref{dFW} are chosen as $\eta_k = \frac{2}{k+2}.$ Then the dFW iterates $\{\mu_k, k \geq 1\}$ satisfy $J(\mu_k) - J^* \leq \frac{2LR^2}{k+2}, k \geq 1$ where $J^* := \inf\left\{J(\mu) : \mu \in \mathcal{M}^+(\mathcal{X},b)\right\}$ and \(R = \sup\{\|\mu_1-\mu_2\| : \mu_1, \mu_2 \in \mathcal{M}^+(\mathcal{X},b)\}\).
\end{theorem}

\begin{algorithm}
\caption{Fully-corrective Frank Wolfe (fc-FW) on $\mathcal{M}^+(\mathcal{X},b)$}\label{alg:fcFW}
\begin{tabular} {l}
  
  \textbf{Input:} Initial measure $\mu_0\in \mathcal{M}^+(\mathcal{X},b)$ \\ 
  \textbf{Output:} Iterates $\mu_1, \dots, \mu_{K}\in \mathcal{M}^+(\mathcal{X},b)$\\
  1 $A_0 \leftarrow \emptyset$ or $\{\mu_0\}$\\
  2 \textbf{for} $k=0,1,\dots, K$ \textbf{do}\\
  3 ~~$x^*(\mu_k)\leftarrow \argmin_{x\in \mathcal{X}} h_{\mu_k}(x)$\\
  4 ~~$A_{k+1}\leftarrow A_k\cup\{b\delta_{x^*(\mu_k)}\}$\\
  5 ~~$\mu_{k+1}\leftarrow \argmin_{\mu\in \mathrm{conv}(A_{k+1})} J(\mu)$\\
  6 \textbf{end for}\\
   
 \end{tabular}
\end{algorithm}

A natural extension of Frank-Wolfe is the fully-corrective version, which often yields improved practical performance~\cite{2013brepik,2017boygeorec}. A fully corrective modification of~\eqref{dFW} is presented in Algorithm~\ref{alg:fcFW}. Unlike~\eqref{dFW}, where the update is a convex combination of the previous iterate \(\mu_k\) and the scaled Dirac measure \( b \delta_{x^*(\mu_k)} \), fully-corrective Frank Wolfe (fc-FW) maintains and optimizes over an expanding set of scaled Dirac measures. Specifically, fc-FW starts with an initial measure \(\mu_0\) and an empty set \(A_0\). At each iteration \(k\), a new scaled Dirac measure \( b \delta_{x^*(\mu_k)} \) is added to \(A_k\), forming the updated set
\[
A_{k+1} = \{ b \delta_{x^*(\mu_0)}, \dots, b \delta_{x^*(\mu_k)} \}.
\]
The fc-FW algorithm then minimizes \( J(\mu) \) over the convex hull of \( A_{k+1} \), incorporating all selected atoms, as shown in Step 5. This fully corrective step is equivalent to solving:

\begin{equation}\label{eq:fcEquivalent}
\min_{p_0, \dots, p_k\in\mathbb{R}} J\left( \sum_{i=0}^{k} p_i b\delta_{x^*(\mu_i)} \right) 
\quad \text{s.t.} \quad \sum_{i=0}^{k} p_i = 1, \quad p_i \geq 0.
\end{equation}
Since \( J \) is convex in $\mu$, \eqref{eq:fcEquivalent} is a (finite-dimensional) convex optimization problem in $p = (p_0,p_1,\ldots,p_k)$.

As an extension of dFW, fc-FW also achieves an \( O(k^{-1}) \) convergence rate under the same assumptions as Theorem~\ref{thm:dFWcomp}, with a similar proof. However, fc-FW introduces additional theoretical properties that enhance its practical performance. In particular, we establish a sufficient decrease property under a weaker assumption, requiring only \( h_{\mu_k}(x^*(\mu_k)) \leq 0 \) rather than the global optimality of \( x^*(\mu_k) \).

\begin{lemma}[Sufficient Decrease]\label{lemma:decrease}
Suppose \(J\) is \(L\)-smooth and \(h_{\mu_k}(x^*(\mu_k)) \leq 0\). Then,
\begin{equation}\label{eq:decrease}
    J(\mu_{k+1}) - J(\mu_k) \leq - \min \left\{ \frac{b}{2LR^2}h_{\mu_k}(x^*(\mu_k))^2, \, \frac{LR^2}{2b} \right\}\leq 0, \quad k \geq 0,
\end{equation}
where \(R = \sup\{\|\mu_1-\mu_2\| : \mu_1, \mu_2 \in \mathcal{M}^+(\mathcal{X},b)\}.\)
\end{lemma}

\begin{proof}{Proof.}
Let \(\mu_{k+0.5} := (1-t_k)\mu_k + t_kb\delta_{x^*(\mu_k)}\), with \(t_k = \min \left\{ \frac{-b}{LR^2}h_{\mu_k}(x^*(\mu_k)), \, 1\right\}\). From the \(L\)-smoothness of \(J\), we can write:
\begin{align}
    J(\mu_{k+0.5}) - J(\mu_k) &\leq J'_{\mu_k}(\mu_{k+0.5} - \mu_k) + \frac{L}{2b}\|\mu_{k+0.5} - \mu_k\|^2 \nonumber \\
    &= t_k J'_{\mu_k}(b\delta_{x^*(\mu_k)} - \mu_k) + \frac{L}{2b}t_k^2\|b\delta_{x^*(\mu_k)} - \mu_k\|^2 \nonumber \\
    &\leq t_k h_{\mu_k}(x^*(\mu_k)) + \frac{L}{2b}t_k^2 R^2 \nonumber \\
    &\leq - \min \left\{ \frac{bh_{\mu_k}(x^*(\mu_k))^2}{2LR^2}, \, \frac{LR^2}{2b} \right\}.
\end{align}

Since \(t_k \in [0,1]\), it follows that \(\mu_{k+0.5} \in \mathrm{conv}(A_{k+1})\). The optimality of \(\mu_{k+1}\) implies that \(J(\mu_{k+1}) \leq J(\mu_{k+0.5})\). Therefore, the inequality~\eqref{eq:decrease} holds.
\hfill \Halmos
\end{proof}

\begin{theorem}[Convergence of fc-FW]\label{thm:convfcFW}  
Suppose \( J \) is \( L \)-smooth, and the influence function satisfies \( h_{\mu_k}(x^*(\mu_k)) \leq 0 \) for all \( k \geq 1 \). Let $J^* := \inf\left\{J(\mu) : \mu \in \mathcal{M}^+(\mathcal{X},b)\right\}$ 
and assume \( J^* > -\infty \). Then, the sequence \( \{J(\mu_k)\}_{k \geq 1} \) generated by fc-FW in Algorithm~\ref{alg:fcFW} is convergent.

Moreover, $\lim_{k \to \infty} h_{\mu_k} (x^*(\mu_k)) = 0$ and $\sum_k |h_{\mu_k}(x^*(\mu_k))| < \infty$.

Finally, there exists an iteration \( k_0 > 0 \) such that \( \left| h_{\mu_{k_0}} (x^*(\mu_{k_0})) \right| < LR^2/b \), and then for all \( n \geq k_0 \),
\begin{equation}
    \min_{1 \leq k \leq n} \left| h_{\mu_k} (x^*(\mu_k)) \right| \leq \frac{1}{\sqrt{n}} \left( \frac{2 L R^2}{b} (J(\mu_1) - J^*) \right)^{\frac{1}{2}},
\end{equation}  
where \(R = \sup\{\|\mu_1-\mu_2\| : \mu_1, \mu_2 \in \mathcal{M}^+(\mathcal{X},b)\}.\) 
\end{theorem}
\begin{proof}{Proof.}
By Lemma~\ref{lemma:decrease}, the sequence \( \{ J(\mu_k) \} \) is decreasing. Furthermore, since \( J(\mu_k) \geq J^* > -\infty \), it is bounded below. Therefore, the sequence \( \{ J(\mu_k) \} \) converges.

Summing both sides of~\eqref{eq:decrease} from \( k=1 \) to \( n \), we obtain
\begin{equation}\label{eq:upbound}
    S_n:=\sum_{k=1}^{n} \min \left\{ h_{\mu_k} (x^*(\mu_k))^2, \left(\frac{LR^2}{b}\right)^2 \right\} 
    \leq \frac{2 L R^2}{b} (J(\mu_1) - J(\mu_{n+1})) 
    \leq \frac{2 L R^2}{b} (J(\mu_1) - J^*).
\end{equation}
Since the sequence \(S_n\)
is increasing and bounded above, it follows that \( S_n \) converges to a finite limit. Consequently, the summands must tend to zero in the limit, that is,
\begin{equation*}
    \lim_{k \to \infty} \min \left\{ h_{\mu_k} (x^*(\mu_k))^2, \left(\frac{LR^2}{b}\right)^2 \right\} = 0.
\end{equation*}
Since \( (LR^2/b)^2 \) is a constant, we conclude that $\lim_{k \to \infty} h_{\mu_k} (x^*(\mu_k)) = 0$,
and there exists some \( k_0 > 0 \) such that \( \left| h_{\mu_{k_0}} (x^*(\mu_{k_0})) \right| < LR^2 / b\). Then, for any \( n \geq k_0 \), we obtain the inequality
\begin{equation*}
    n \, \min_{1 \leq k \leq n} h_{\mu_k} (x^*(\mu_k))^2 
    \leq \sum_{k=1}^{n} \min \left\{ h_{\mu_k} (x^*(\mu_k))^2, \left(\frac{LR^2}{b}\right)^2 \right\}.
\end{equation*}
Applying~\eqref{eq:upbound}, we derive
\begin{equation*}
    \min_{1 \leq k \leq n} \left| h_{\mu_k} (x^*(\mu_k)) \right| 
    \leq \frac{1}{\sqrt{n}} \left( \frac{2 L R^2}{b} (J(\mu_1) - J^*) \right)^{\frac{1}{2}}, \quad \forall n \geq k_0.
\end{equation*}
This proves the theorem. \hfill \Halmos
\end{proof}

The results in Lemma~\ref{lemma:decrease} and Theorem~\ref{thm:convfcFW} highlight important theoretical properties of fc-FW. Below, we summarize the key implications:

\begin{itemize}
    \item[(a)] Unlike dFW, fc-FW does not require solving the subproblem in Step 3 of Algorithm~\ref{alg:fcFW} to global optimality. Instead, it suffices to find a local minimizer \( x^*(\mu_k) \) such that \( h_{\mu_k}(x^*(\mu_k)) \leq 0 \). This is a significant advantage in practice, especially when \( h_{\mu}(x) \) is nonconvex and global optimization is computationally infeasible. The ability to use a good local minimizer enhances the algorithm's scalability and practical applicability.
    \item[(b)] The theoretical behavior of fc-FW closely resembles that of gradient descent (GD) in Euclidean spaces. In both methods, a first-order object guides the updates— the gradient in GD and the influence function \( h_{\mu}(x) \) in fc-FW. Under \( L \)-smoothness, both methods ensure a decrease in the objective value at each iteration. Lemma~\ref{lemma:decrease} establishes this for fc-FW, analogous to the standard descent property in GD. In both methods, global convergence is not guaranteed. In GD, the norm of the gradient converges to zero as the algorithm progresses. Similarly, Theorem~\ref{thm:conv} guarantees that \( h_{\mu_k}(x^*(\mu_k)) \to 0 \) at a rate of \( O(1/\sqrt{k}) \), mirroring the behavior of the gradient norm. Moreover, when fc-FW is analyzed under the same assumptions as Theorem~\ref{thm:dFWcomp}, it maintains an overall function value convergence rate of \( O(k^{-1}) \), just as gradient descent does in convex optimization. These similarities help build intuition about the theoretical behavior of fc-FW.
\end{itemize}

\section{PROBLEM REGULARITY.}\label{sec:regularity}

The optimization problem \eqref{mainopt} involves minimizing the functional $J(\cdot)$ over the set $\mathcal{M}^+(\mathbb{R}^2,b)$. Recall, however, that emergency incidents occur only within the region \( \mathcal{S} \subset \mathbb{R}^2 \). It is natural to consider whether restricting the optimization domain to \( \mathcal{M}^+(\mbox{cvx}(\mathcal{S}),b) \) is sufficient. Here, \( \mbox{cvx}(\mathcal{S}) \) denotes the closure of the convex hull of \( \mathcal{S} \): $$\mbox{cvx}(\mathcal{S}) := {\rm{cl}}\left(\bigcap_{\alpha}\{\mathcal{S}_{\alpha} \subseteq \mathbb{R}^2 \mbox{ is convex  and } \mathcal{S}_{\alpha} \supseteq \mathcal{S}\}\right).$$ A key intuition behind this restriction is that placing a volunteer outside $\mbox{cvx}(\mathcal{S})$ should not be optimal --- since all incident locations \( y \) lie within \( \mathcal{S} \), projecting any mass outside \( \mbox{cvx}(\mathcal{S}) \) onto its boundary should reduce the response time to all incidents, thereby decreasing or preserving the death rate. This implies that restricting the feasible set to measures supported on \( \mbox{cvx}(\mathcal{S}) \) should not eliminate optimal solutions, motivating the restricted problem
\begin{align}\label{restrictedopt} \mbox{min.} & \quad J(\mu) = \int_{\mathcal{S}} \eta(\mbox{d}y) \int_0^\infty \exp\Big\{-\mu(\bar{B}(y,t))\Big\} \, \mathop{d\death(t)}, \nonumber \\ \mbox{s.t.} & \quad \mu \in \mathcal{M}^+(\mbox{cvx}(\mathcal{S}),b), \tag{$P$}\end{align} where $\bar{B}(y,t):= B(y,t) \,\, \cap \,\, \mathcal{S}.$ The following theorem formalizes this result. Let int$(A)$ and $\partial(A)$ denote the interior and boundary of the set $A$ respectively.

\begin{theorem}[Restricted Problem Sufficiency]\label{thm:restrict} For each $\nu \in \mathcal{M}^+(\mathbb{R}^2,b)$, define the restricted measure \begin{align*}\label{pushforward}  \tilde{\nu}(A) := \nu\left(A \, \cap \, \rm{int}\left(\rm{cvx}(\mathcal{S})\right)\right) +  \nu\left(\Pi^{-1}_{\rm{cvx}(\mathcal{S})}\left(A \, \cap \, \partial(\rm{cvx}(\mathcal{S}))\right)\right), \quad A \in \mathcal{B}(\mathbb{R}^2). \end{align*} Then, $\tilde{\nu} \in \mathcal{M}^+(\emph{\mbox{cvx}}(\mathcal{S}),b),\mbox{ and } J(\tilde{\nu}) \leq J(\nu).$\end{theorem} 

\begin{proof}{Proof.}
Since $\nu \in \mathcal{M}^+(\mathbb{R}^2,b)$, by Definition~\ref{def:projection}  \begin{equation*}
     \tilde{\nu}\left(\rm{cvx}(\mathcal{S})\right) = \nu\left(\Pi^{-1}_{\rm{cvx}(\mathcal{S})}\left(\rm{cvx}(\mathcal{S})\right)\right)=\nu\left(\mathbb{R}^2 \right)=b,
\end{equation*} that is $\tilde{\nu}$ is in $\mathcal{M}^+(\mbox{cvx}(\mathcal{S}),b)$. For any $ A \in \mathcal{B}(\mathbb{R}^2)$, $\Pi^{-1}_{\rm{cvx}(\mathcal{S})}\left(A \, \cap \, \partial(\rm{cvx}(\mathcal{S}))\right) \supseteq A \, \cap \, \partial(\rm{cvx}(\mathcal{S}))$.
Then, for $\forall A \in \mathcal{B}(\mathcal{S})$,
\begin{align} \label{tildeVgeqV}
    \tilde{\nu}\left(A \right)&=\nu\left(A \, \cap \, \rm{int}\left(\rm{cvx}(\mathcal{S})\right)\right) +  \nu\left(\Pi^{-1}\left(A \, \cap \, \partial(\rm{cvx}(\mathcal{S}))\right)\right) \nonumber \\ &\geq \nu\left(A \, \cap \, \rm{int}\left(\rm{cvx}(\mathcal{S})\right)\right) +  \nu\left(A \, \cap \, \partial(\rm{cvx}(\mathcal{S}))\right) = \nu\left(A \right).
\end{align}
Now, according to \eqref{tildeVgeqV},
\begin{align*}
    J(\tilde{\nu})& =\int_{\mathcal{S}} \eta(\mbox{d}x) \int_0^{\infty} \exp\Big\{-\tilde{\nu}(\bar{B}(x,t))\Big\} \, \mbox{d}\death(t) \\ & \le \int_{\mathcal{S}} \eta(\mbox{d}x) \int_0^{\infty} \exp\Big\{-\nu(\bar{B}(x,t))\Big\} \, \mbox{d}\death(t) = J(\nu),
\end{align*}
where $\bar{B}(x,t)= B(x,t) \,\, \cap \,\, \mathcal{S}$ is a set in $ \mathcal{B}(\mathcal{S})$, proving the assertion.   \hfill \Halmos
\end{proof}

Next, we state a standard result that the set of measures $\mathcal{M}^+(\mathrm{cvx}(\mathcal{S}), b)$ is convex and is compact under the weak$^*$ topology if the set $\mathcal{S}$ is compact. These properties aid the solution of Problem~\eqref{restrictedopt}. 

\begin{theorem}[Convexity and Compactness of Feasible Region]\label{thm:feasibleregion}  The following assertions about the set $\mathcal{M}^+(\mbox{cvx}(\mathcal{S}),b)$ hold. \begin{enumerate} \item[(a)] $\mathcal{M}^+(\mbox{cvx}(\mathcal{S}),b)$ is convex, that is, for $\mu_1, \mu_2 \in \mathcal{M}^+(\emph{\mbox{cvx}}(\mathcal{S}),b)$ and $\alpha \in [0,1]$, $$\alpha \mu_1 + (1-\alpha)\mu_2 \in \mathcal{M}^+(\emph{\mbox{cvx}}(\mathcal{S}),b).$$ \item[(b)] If $\mathcal{S} \subset \mathbb{R}^2$ is compact, then $\mathcal{M}^+(\mbox{cvx}(\mathcal{S}),b)$ is compact under the weak$^*$ topology, and hence, for each sequence $\{\mu_n, n \geq 1\}, \mu_n \in \mathcal{M}^+(\emph{\mbox{cvx}}(\mathcal{S}),b)$ there exists a convergent subsequence with limit in $\mathcal{M}^+(\emph{\mbox{cvx}}(\mathcal{S}),b)$.\end{enumerate}
\end{theorem}

We next demonstrate that the functional $J: \mathcal{M}^+(\mbox{cvx}(\mathcal{S}),b) \to \mathbb{R}^+$ is convex on $\mathcal{M}^+(\mbox{cvx}(\mathcal{S}),b)$. 
\begin{theorem}[Convexity of Objective]\label{thm:conv} The functional $J$ defined in~\eqref{objexp} is convex on $\mathcal{M}^+(\mathrm{cvx}(\mathcal{S}),b)$, that is, for $\mu_1, \mu_2 \in \mathcal{M}^+(\mbox{cvx}(\mathcal{S}),b)$ and $\alpha \in (0,1)$, $$J(\alpha \mu_1 + (1-\alpha)\mu_2) \leq \alpha J(\mu_1) + (1-\alpha) J(\mu_2).$$ \end{theorem} \begin{proof}{Proof.} Since $e^{-x}$ is convex in $\mathbb{R}$, $\forall \alpha \in [0,1]$
\begin{align} 
    J(\alpha \mu_1 + (1-\alpha)\mu_2)  &= \int_{\mathcal{S}} \eta(\mbox{d}x) \int_0^\infty \exp\Big\{-\alpha \mu_1(\bar{B}(x,t)) - (1-\alpha)\mu_2(\bar{B}(x,t))\Big\} \, \mbox{d}\death(t) \nonumber\\
    & \leq \int_{\mathcal{S}} \eta(\mbox{d}x) \int_0^\infty \alpha \exp \{- \mu_1(\bar{B}(x,t))\} +  (1-\alpha)\exp \{- \mu_2(\bar{B}(x,t))\} \, \mbox{d}\death(t) \nonumber\\ & = \alpha J(\mu_1) + (1-\alpha) J(\mu_2) \nonumber
\end{align} proving the assertion of the theorem. \hfill \Halmos
\end{proof}  

Since we have shown through Theorem~\ref{thm:feasibleregion} that the feasible region $\mathcal{M}(\mbox{cvx}(\mathcal{S}),b)$ is compact, and through Theorem~\ref{thm:conv} that the objective function is convex, it seems intuitively clear that Problem~\eqref{restrictedopt} is well-defined in the sense of solution existence. The following theorem makes this precise.

\begin{theorem}[Solution Existence]\label{thm:existence} Suppose the set $\mathcal{S} \subset \mathbb{R}^2$ is compact. Then the infimum associated with Problem~\eqref{restrictedopt} is attained, that is there exists $\mu^* \in \mathcal{M}^+(\rm{cvx}(\mathcal{S}),b)$ such that $$J(\mu^*) = \inf_{\mu \in \mathcal{M}^+(\rm{cvx}(\mathcal{S}),b)} J(\mu).$$

\end{theorem}
\begin{proof}{Proof.} 

Let \( J^* := \inf J(\mu) \). By definition, there exists a sequence of measures \( \{\mu_n\}_{n \geq 1} \subset \mathcal{M}^+(\mathrm{cvx}(\mathcal{S}),b) \) such that $\lim_{n\to\infty}J(\mu_n) = J^*.$ According to the compactness of \( \mathcal{M}^+(\mathrm{cvx}(\mathcal{S}),b) \) by Theorem~\ref{thm:feasibleregion}, we can extract a convergent subsequence \( \{\mu_{n_k}\}_{k \geq 1} \) with $
\mu_{n_k} \to \tilde{\mu}$  in the weak topology, and $\tilde{\mu} \in \mathcal{M}^+(\mathrm{cvx}(\mathcal{S}),b)$.
We know \( J \) is convex by Theorem 3, implying that \( J \) is lower semicontinuous in the weak topology. Therefore, we obtain $J^* = \liminf_{k \to \infty} J(\mu_{n_k}) \geq J(\tilde{\mu}).$
By the definition of the infimum, we also have $J^* \leq J(\tilde{\mu})$, which implies $J(\tilde{\mu}) = J^*.$
Since \( \tilde{\mu} \in \mathcal{M}^+(\mathrm{cvx}(\mathcal{S}),b) \), this proves the assertion.
\hfill \Halmos
\end{proof}


As alluded to in Section~\ref{sec:fw_methods}, the existence of the von Mises derivative is central to implementation, since it affords the ability to perform a ``particle update'' of the type seen in Step 3 of Algorithm~\ref{alg:fcFW}. Does the von Mises derivative (of the objective) exist in the context of~\eqref{restrictedopt}? The following theorem provides an answer in the affirmative. Importantly, it provides expressions for the von Mises derivative and the influence function in a form that is used in subsequent sections for implementation. 

\begin{theorem}[Influence Function and von Mises Derivative]\label{thm:vonMises} The influence function $h_{\mu}: \mathrm{cvx}(\mathcal{S}) \to \mathbb{R}$ is given by \begin{equation}\label{eq:influence}
    h_{\mu}(x) = \int_{\mathcal{S}} \eta(\mathop{dy}) \int_0^{\infty} \left(\mu\left(\bar{B}(y,t) \right) - b\,\mathbb{I}_{[\|y-x\|,\infty)}(t)\right) \exp\left\{-\mu\left(\bar{B}(y,t) \right) \right\} \, \mathop{d\death(t)}.
\end{equation}
Furthermore, \begin{equation}\label{vonMisesexp} \lim_{t \to 0^+} \frac{1}{t} \left( J(\mu + t(\nu - \mu)) - J(\mu) \right) = \int_{\mathcal{S}} \eta(\mathop{dy}) \int_0^{\infty} \exp\left\{-\mu\left(\bar{B}(y,t) \right) \right\} \left( \mu - \nu\right)\left(\bar{B}(y,t) \right) \, \mathop{d\death(t)},\end{equation} implying that $$J'_{\mu}(\nu-\mu) = \lim_{t \to 0^+} \frac{1}{t} \left( J(\mu + t(\nu - \mu)) - J(\mu) \right)  = \frac{1}{b}\mathbb{E}_{x \sim \nu}\left[ h_{\mu}(x)\right].$$ 
\end{theorem}

\begin{proof}{Proof.}
From the expression in~\eqref{objexp}, we have \begin{equation}\label{objexp2}
    J(\mu) = \int_{\mathcal{S}} \eta(\mbox{d}y) \int_0^\infty \exp\Big\{-\mu(\bar{B}(y,t))\Big\} \, \mathop{d\death(t)}.
\end{equation}

Now notice that
\begin{align}\label{eq:vonmises}
    \lefteqn{\lim_{s \to 0^+} \frac{J(\mu + s(\nu - \mu)) - J(\mu)}{s}} \nonumber \\
    & = \lim_{s \to 0^+} \int_{\mathcal{S}} \eta(\mbox{d}y) \int_0^\infty \exp \Big\{-\mu(\bar{B}(y,t))\Big\} \frac{\exp\Big\{ s(\mu - \nu)(\bar{B}(y,t))\Big\}-1}{s} \, \mathop{d\death(t)} \nonumber \\ 
& = \int_{\mathcal{S}} \eta(\mbox{d}y) \int_0^\infty \exp\Big\{-\mu(\bar{B}(y,t))\Big\} (\mu - \nu)(\bar{B}(y,t)) \, \mathop{d\death(t)},\end{align} where the last equality follows from Lebesgue's dominated convergence theorem. 

We see by setting $\nu = b\,\delta_x$ in~\eqref{eq:vonmises} that
\begin{align} \label{infproof}
h_{\mu}(x) &= \lim_{s \to 0^+} \frac{1}{s} (J(\mu + s(b\,\delta_x - \mu)) - J(\mu)) \nonumber \\
&=  \int_{\mathcal{S}} \eta(\mathop{dy}) \int_0^{\infty} \left(\mu\left(\bar{B}(y,t) \right) - b\,\mathbb{I}_{[\|y-x\|,\infty)}(t)\right) \exp\left\{-\mu\left(\bar{B}(y,t) \right) \right\} \, \mathop{d\death(t)}.
\end{align}
Furthermore, notice from~\eqref{infproof} that
\begin{align} \int_{x \in \mathbb{R}^2} \frac{1}{b}h_{\mu}(x) \, (\nu-\mu)(\mbox{d}x) =  \int_{\mathcal{S}} \eta(\mbox{d}y) \int_0^\infty \exp\Big\{-\mu(\bar{B}(y,t))\Big\} (\mu - \nu)(\bar{B}(y,t)) \, \mathop{d\death(t)},
\end{align}
which is the right-hand side of~\eqref{eq:vonmises}, implying the existence of the von Mises derivative $J'_{\mu}(\nu-\mu)$. 

\hfill \Halmos
\end{proof}
Having established the differentiability properties of \( J \), we now adapt two key results from~\cite{yu2024det} for the context of solving~\eqref{restrictedopt}. The first provides necessary and sufficient conditions for optimality in terms of the influence function \( h_{\mu}(x) \), and the second describes the structure of the support of optimal measures.
\begin{theorem}[Conditions for Optimality~\cite{yu2024det}]\label{thm:optCondition} The measure $\mu^*$ is optimal for problem~\eqref{restrictedopt}, that is, $J(\nu) \geq J(\mu^*)$ $\forall \nu \in \mathcal{M}^+(\mathrm{cvx}(\mathcal{S}),b)$ if and only if $h_{\mu^*}(x) \geq 0$ for all $x \in \mathrm{cvx}(\mathcal{S})$, where the influence function $h_{\mu^*}$ at $\mu^*$ is given by Theorem~\ref{thm:vonMises}.
\end{theorem}
\begin{lemma}[Support of Optimal Measure~\cite{yu2024det}]\label{lem:supportopt} If the measure $\mu^*$ is optimal for problem~\eqref{restrictedopt}, then the support of $\mu^*$ satisfies \begin{equation}\label{suppinfl}
    \mathrm{supp}(\mu^*) \subseteq \left\{ x \in \mathrm{cvx}(\mathcal{S})\,:\,h_{\mu^*}(x)=0 \right\} \quad \mu^* \mbox{ a.s.\ }
\end{equation} Moreover, if $h_{\mu^*}$ is differentiable, we have  \begin{equation}\label{suppdinfl}
    \mathrm{supp}(\mu^*) \subseteq \left\{ x \in \mathrm{cvx}(\mathcal{S})\,:\,\nabla h_{\mu^*}(x)=0 \right\} \quad \mu^* \mbox{ a.s.\ }
\end{equation}
\end{lemma}

\section{Two Simple Cases for Intuition}\label{sec:specialCases}
While Theorem~\ref{thm:existence} guarantees the existence of an optimal solution to problem~\eqref{restrictedopt}, the structure of optimal measures can be nontrivial even in seemingly simple cases. Consider the setting where out-of-hospital cardiac arrests (OHCA) occur at a finite number of fixed demand points, \(\mathcal{S} = \{y_1, y_2, \dots, y_n \}\), with the incident distribution given by
\begin{equation*}
    \eta = \sum_{i=1}^n \lambda_i \delta_{y_i}, \quad \sum_{i=1}^n \lambda_i = 1, \quad \lambda_1, \dots, \lambda_n \geq 0.
\end{equation*}
While the case \( n = 2 \) admits an analytical solution with the optimal measure supported on the demand points (see below), the structure becomes significantly more intricate for \( n = 3 \), where the optimal measure is no longer restricted to the demand points, that is, \( \mathrm{supp}(\mu^*) \not\subseteq \{y_1, \dots, y_n \} \) for $n \geq 3$.


\subsection{Two Demand Points \texorpdfstring{$\mathbf{(n=2)}$}{(n=2)}}\label{sec:casen=2}
We can derive the analytic solution to problem~\eqref{restrictedopt} for the case with two demand points $y_1, y_2$ ($n=2$). 
        
        
\begin{theorem}\label{thm:opt2dem}
    Suppose the probability measure of event arrivals is \(\eta = \lambda_1\,\delta_{y_1} + \lambda_2\,\delta_{y_2}\), with \(\lambda_1 + \lambda_2 = 1\) and \(\lambda_1, \, \lambda_2 > 0\). Then
\begin{equation}\label{opt2Dem}
    \mu^* = \alpha_1 \, \delta_{y_1} + \alpha_2 \, \delta_{y_2},
\end{equation}
where
\[
\alpha_1 =
\begin{cases}
0, & 0 \leq \frac{\lambda_1}{\lambda_2} < e^{-b}, \\
\frac{b}{2} + \frac{1}{2} \log \frac{\lambda_1}{\lambda_2}, & e^{-b} \leq \frac{\lambda_1}{\lambda_2} \leq e^b, \\
b, & \frac{\lambda_1}{\lambda_2} > e^b,
\end{cases}
\]
and \(\alpha_2 = b - \alpha_1\), is an optimal measure for problem~\eqref{restrictedopt}.
\end{theorem}
\begin{proof}{Proof.}
    With our specific choice of $\mu^*$, for \(i=1,2\),
    \begin{equation}\label{eq:ballmeasure2Dem}
        \mu^*\left(\bar{B}(y_i,t) \right) = 
        \begin{cases}
            \alpha_i, & 0 \leq t < \|y_1 - y_2\|, \\
            b, & t \geq \|y_1 - y_2\|.
        \end{cases}
    \end{equation}
    When \(e^{-b} \leq \frac{\lambda_1}{\lambda_2} \leq e^b\), for any \(x \in \mathrm{cvx}(\{y_1,\,y_2\})\), from Theorem~\ref{thm:vonMises}, the influence function at \(\mu^*\) is
    \begin{align*}
        h_{\mu^*}(x) &= \sum_{i=1}^2 \lambda_i \int_0^{\infty} 
        \left( \mu^*\left(\bar{B}(y_i,t) \right) - b \, \mathbb{I}_{[\|x-y_i \|,\infty)}(t) \right) 
        \exp\left\{-\mu^*\left(\bar{B}(y_i,t) \right) \right\} \, \mathop{d\death(t)} \\ 
        &= \sum_{i=1}^2 \lambda_i \left( 
        \alpha_i e^{-\alpha_i} \big[\death(\|y_1 - y_2 \|) - \death(0)\big] 
        - b e^{-\alpha_i} \big[\death(\|y_1 - y_2 \|) - \death(\|x-y_i \|)\big] 
        \right) \\  
        &= \sqrt{\lambda_1 \lambda_2} b e^{-\frac{b}{2}} \left( 
        \death(\|x-y_1 \|) + \death(\|x-y_2 \|) - \death(\|y_1-y_2 \|) - \death(0) 
        \right) \\ 
        &\geq 0.
    \end{align*}
    The second equality follows from substituting equation~\eqref{eq:ballmeasure2Dem} into the influence function. The final inequality arises from the concavity of \(\death\) and the fact that \(\|y_1 - y_2 \| = \|x - y_1 \| + \|x - y_2 \|\).
    Similarly, we can prove that \(h_{\mu^*}(x) \geq 0\) for the cases where \(\alpha_1 = 0\) or \(\alpha_1 = b\). According to the optimality condition in Theorem~\ref{thm:optCondition}, \(\mu^*\) is the optimal solution. 
    \hfill \Halmos
\end{proof}

Theorem~\ref{thm:opt2dem} provides the analytical solution to problem~\eqref{restrictedopt} for the case of two demand points. The optimal measure is supported on the demand points, that is, \(\mathrm{supp}(\mu^*) \subseteq \{y_1,\, y_2\}\). 

\subsection{Three Demand Points \texorpdfstring{$\mathbf{(n=3)}$}{(n=3)}}\label{sec:casen=3}
Observing the optimal measure for \(n=2\), a natural question arises: do optimal measures for problem~\eqref{restrictedopt} generally share the structure of being supported only on the demand points for general \(n > 2\)? We now show that the answer to this question is in the negative using an example with three demand points.

\begin{figure}[htbp]
    \centering
    \begin{tikzpicture}
        \coordinate (y1) at (0, 0); 
        \coordinate (y2) at (4, 0); 
        \coordinate (y3) at (2, {2*sqrt(3)}); 
        \coordinate (o) at (2, {2/sqrt(3)}); 

        \draw[thick] (y1) -- (y2) -- (y3) -- cycle;

        \filldraw[black] (y1) circle (2pt) node[below] {\(y_1\)};
        \filldraw[black] (y2) circle (2pt) node[below] {\(y_2\)};
        \filldraw[black] (y3) circle (2pt) node[above] {\(y_3\)};
        \filldraw[black] (o) circle (2pt) node[right] {\(o\)};
    \end{tikzpicture}
    \caption{Three symmetric demand points.}
    \label{fig:demandpts3}
\end{figure}
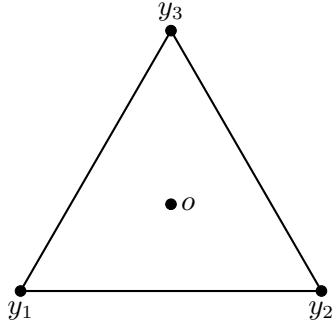

Consider a symmetric setup as in Figure~\ref{fig:3dempts} where the incident measure is:
\begin{equation*}
    \eta = \frac{1}{3} \delta_{y_1} + \frac{1}{3} \delta_{y_2} + \frac{1}{3} \delta_{y_3},
\end{equation*}
and the demand points are located at the vertices of an equilateral triangle: \(y_1 = (0,0)\), \(y_2 = (1,0)\), and \(y_3 = \left(1/2, \sqrt{3}/2\right)\). Due to the symmetry of the setup, if the support of the optimal measure is restricted to \(\{y_1, y_2, y_3\}\), the measure must take the form:
$\tilde{\mu} = \frac{b}{3} \delta_{y_1} + \frac{b}{3} \delta_{y_2} + \frac{b}{3} \delta_{y_3}.$ Given \(\|y_1 - y_2\| = \|y_1 - y_3\| = \|y_2 - y_3\| = 1\), the corresponding influence function at \(\tilde{\mu}\) is:
\begin{align}
    h_{\tilde{\mu}}(x) &= \frac{1}{3}\sum_{i=1}^3 \int_0^{\infty} 
    \left( \tilde{\mu}\left(\bar{B}(y_i,t) \right) - b \, \mathbb{I}_{[\|x-y_i \|,\infty)}(t) \right) 
    \exp\left\{-\tilde{\mu}\left(\bar{B}(y_i,t) \right) \right\} \, \mathrm{d}\death(t) \nonumber \\ 
    &= \frac{b}{3} e^{-\frac{b}{3}} \left( 
        \death(\|x-y_1 \|) + \death(\|x-y_2 \|) + \death(\|x-y_3 \|) - 2\death(1) - \death(0) \right).
\end{align} Using the probability of survival function from~\cite{waalewijn2001survival,van2024modeling}, $\death(t) = 1 - \left(1 + e^{0.679 + 0.262t}\right)^{-1},$ we compute the influence value at the centroid \(x = o = (1/2, 1/(2\sqrt 3))\) as $h_{\tilde{\mu}}(o) = \frac{1}{3} e^{-\frac{b}{3}} (-0.0129) < 0.$ According to the optimality condition in Theorem~\ref{thm:optCondition}, \(\tilde{\mu}\) is not an optimal solution. Therefore, the support of the optimal measure is not restricted to the demand points, that is, \(\mathrm{supp}(\mu^*) \not\subseteq \{y_1, y_2, y_3\}\). This indicates that the optimal measure \(\mu^*\) has a more complex structure for the three-demand-point case and, by extension, for general cases. In general, deriving analytical solutions is challenging, making numerical methods essential for accurately computing optimal solutions.

In the following section, we adapt the FW methods introduced in Section~\ref{sec:fw_methods} to solve problem~\eqref{restrictedopt}. Specifically, we apply fc-FW (Algorithm~\ref{alg:fcFW}) to the three-demand-point case and extend the approach to a similar four-demand-point case. The results presented in Figure~\ref{fig:3dempts} confirm the intricate structure of the optimal measure, with further discussion in Section~\ref{sec:3dempts&4dempts}.

\section{Numerical Illustration}\label{sec:implementation}

As hinted in Section~\ref{sec:specialCases}, the problem in~\eqref{restrictedopt} defies analytic solution even in simple cases, calling for an iterative algorithm such as fc-FW. Furthermore, and as we shall demonstrate in this section, the solution to~\eqref{restrictedopt} is often remarkably intricate but with patterns that are plausible post-hoc. In what follows, we begin with an illustration of the solution to discrete cases from Section~\ref{sec:specialCases} with \( n = 3 \) and \( n = 4 \). This is followed by a more realistic scenario where \( \eta \) follows a continuous distribution, better approximating real-world emergency incidents. Finally, we apply fc-FW to a full-scale emergency response problem in the city of Auckland, New Zealand, demonstrating its practical applicability in optimizing emergency resource allocation. In all experiments, we adopt the probability of death function \( \death(t) = 1 - (1 + e^{0.679 + 0.262t})^{-1} \), following the survival analysis in~\cite{de2003optimal}, which models the relationship between response time and survival rates.


\subsection{Implementing fc-FW}
In the emergency response setting, the fc-FW algorithm in Algorithm~\ref{alg:fcFW} is applied with $\mathcal{X} = \mathrm{cvx}(\mathcal{S})$. As established in Theorem~\ref{thm:vonMises}, the influence function \( h_{\mu}(x) \) is given by:
\begin{align}
    h_{\mu}(x) &= \int_{\mathcal{S}} \eta(\mbox{d}y) \int_0^{\infty} \left(\mu\left(\bar{B}(y,t) \right) - b \, \mathbb{I}_{[\|x-y \|,\infty)}(t)\right) \exp\left\{-\mu\left(\bar{B}(y,t) \right) \right\} \, \mathop{d\death(t)} \nonumber\\
    & =\int_{\mathcal{S}} \eta(\mbox{d}y) \int_0^{\|x-y \|}  b\,\exp\left\{-\mu\left(\bar{B}(y,t) \right) \right\} \, \mathop{d\death(t)} + c,
\end{align} where \( c \) is a constant independent of \( x \), ensuring that only the first term affects optimization. This function serves as the first-order object and plays a crucial role in the iterative updates of the fc-FW algorithm introduced in Section~\ref{sec:fw_methods}. To implement Algorithm~\ref{alg:fcFW} in practice, we need efficient methods to solve two key subproblems at each iteration. The first involves finding the next support point by solving
\begin{equation} \label{eq:subproblem1} x^*(\mu_k) = \argmin_{x\in \mathrm{cvx}(\mathcal{S})} h_{\mu_k}(x), \end{equation}
which identifies the most influential location to update the measure. The second subproblem updates the measure by solving
\begin{equation} \label{eq:subproblem2} \mu_{k+1} = \argmin_{\mu\in \mathrm{conv}(A_{k+1})} J(\mu), \end{equation}
redistributing mass among the selected points. As discussed in Section~\ref{sec:fw_methods}, the measure update step~\eqref{eq:subproblem2} is equivalent to a finite-dimensional convex optimization problem~\eqref{eq:fcEquivalent}. Standard convex optimization methods, such as projected gradient descent or projected stochastic gradient descent, can be used to efficiently compute \( \mu_{k+1} \). On the other hand, solving~\eqref{eq:subproblem1} is more challenging because \( h_{\mu}(x) \) is generally nonconvex, although our experience has been that emergency response problems involve influence functions that can be minimized efficiently even if nonconvex. During implementation, we use a gradient-based optimization algorithm such as Adam~\cite{kingma2017adam}, initialized from a random point uniformly sampled from $\mathrm{cvx}(\mathcal{S})$, to minimize \( h_{\mu}(x) \), allowing us to obtain high-quality solutions while maintaining computational efficiency.

  
   


\subsection{fc-FW on Three and Four Demand points.}\label{sec:3dempts&4dempts}

\begin{figure}[htb]
\FIGURE{
\parbox{\textwidth}{
\centering
    \subcaptionbox{Optimal measure $\hat{\mu}^*$ with color representing mass.\label{subfig:mu3demfcFW}}{
        \includegraphics[width=0.48\textwidth]{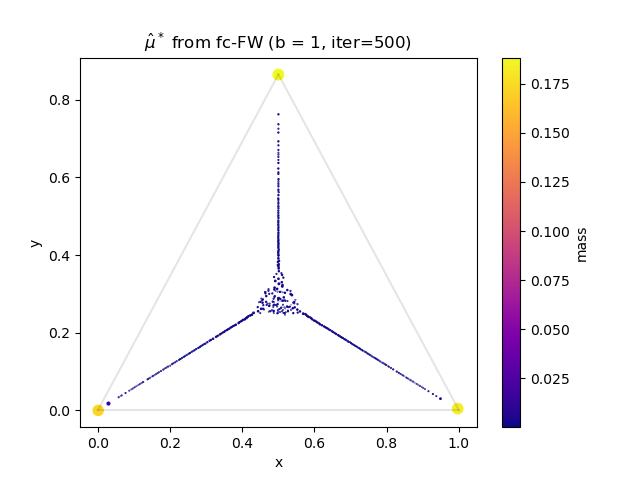}
    }
    \hfill
    \subcaptionbox{Contour plot of the influence function $h_{\hat{\mu}^*}(x)$.\label{subfig:influe3demfcFW}}{
        \includegraphics[width=0.48\textwidth]{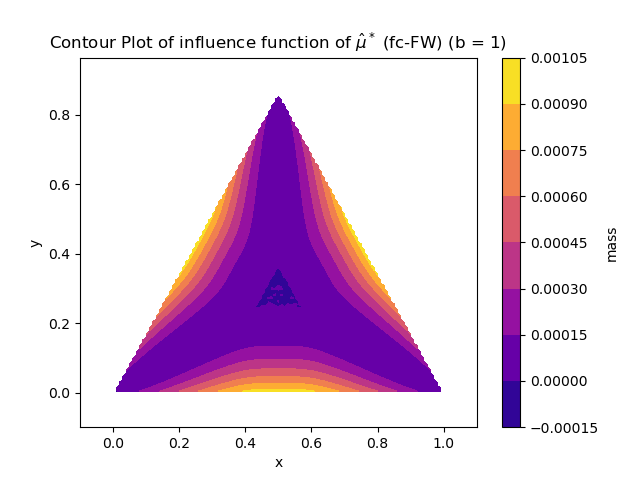}
    }
        \\[1ex]
    \subcaptionbox{Optimal measure $\hat{\mu}^*$ with color representing mass.\label{subfig:mu4demfcFW}}{
        \includegraphics[width=0.48\textwidth]{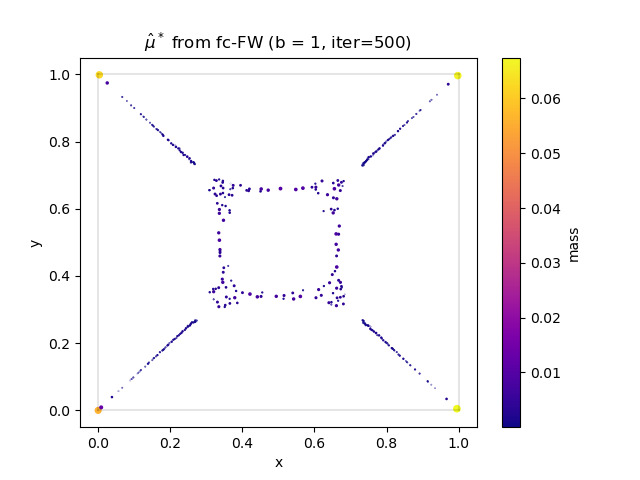}
    }
    \hfill
    \subcaptionbox{Contour plot of the influence function $h_{\hat{\mu}^*}(x)$.\label{subfig:influe4demfcFW}}{
        \includegraphics[width=0.48\textwidth]{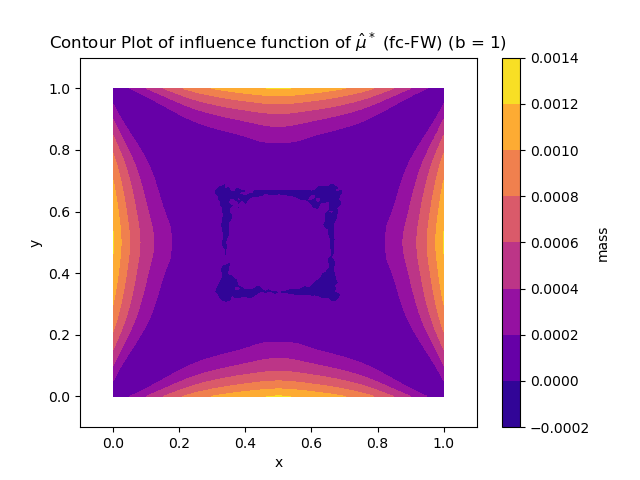}
    }
    }
}
{
Emergency Response with Three and Four Demanding Points
\label{fig:3dempts}}
{
(a)–(d) show the results of applying the fc-FW algorithm (Algorithm~\ref{alg:fcFW}) with total volunteer mass \( b = 1 \), after 500 iterations. (a) and (c) display the optimal volunteer measure \( \hat{\mu}^* \) for three and four demand points, respectively, with color indicating assigned mass. (b) and (d) show the corresponding contour plots of the influence function \( h_{\hat{\mu}^*}(x) \), which confirm approximate optimality by remaining nearly nonnegative throughout \( \mathrm{cvx}(\mathcal{S}) \).
}
\end{figure}

In Section~\ref{sec:specialCases}, we analyzed special cases where incidents follow a discrete distribution, \( \eta = \sum_{i=1}^{n} \lambda_i \delta_{y_i} \). For \( n = 2 \), the optimal measure follows intuition, with support concentrated at the demand points. However, for \( n = 3 \) and \( n = 4 \), the structure becomes more complex, and the optimal measure is no longer restricted to the demand points.

We first examine the three-demand-point case, following the setup in Section~\ref{sec:casen=3}, where the possible incidents are positioned at the vertices of an equilateral triangle:
$y_1 = (0,0)$, $y_2 = (1,0)$ and $y_3 = \left(\frac{1}{2}, \frac{\sqrt{3}}{2} \right)$. The total volunteer mass is set to \( b = 1 \). Using fc-FW (Algorithm~\ref{alg:fcFW}) for 500 iterations, we obtain the optimal volunteer measure \( \hat{\mu}^* \), shown in Figure~\ref{subfig:mu3demfcFW}. The results confirm a nontrivial solution with a complex structure, consistent with the discussion in Section~\ref{sec:casen=3}, where the optimal measure is not simply concentrated on the demand points. The corresponding contour plot of the influence function \( h_{\hat{\mu}^*}(x) \) in Figure~\ref{subfig:influe3demfcFW} verifies the optimality condition from Theorem~\ref{thm:optCondition}, which requires \( h_{\hat{\mu}^*}(x) \geq 0 \) for all \( x \in \mbox{cvx}(\mathcal{S}) \). The minimum value of \( h_{\hat{\mu}^*}(x) \) is above \(-0.00015\), which, given numerical approximations and optimization tolerances, supports the contention that the solution is approximately globally optimal.

Next, we extend this analysis to the four-demand-point case to confirm the emergence of nontrivial optimal solutions with complex structure. The demand points are arranged in a square formation:
\[
y_1 = (0,0), \quad y_2 = (1,0), \quad y_3 = (0,1), \quad y_4 = (1,1).
\]
As in the three-point case, we use fc-FW to compute the optimal measure, shown in Figure~\ref{subfig:mu4demfcFW}, along with its corresponding influence function in Figure~\ref{subfig:influe4demfcFW}. The results again show that the optimal measure is not simply concentrated at the demand points but is spread out in a more intricate pattern. The influence function approximately satisfies the optimality condition in Theorem~\ref{thm:optCondition}, confirming that the computed measure is approximately globally optimal. This reinforces the observation that the optimal allocation strategy does not necessarily align with intuitive placements, at least not ex ante, highlighting the need for efficient computational methods to determine precise solutions.


\subsection{fc-FW on Uniform and Mixture-of-Uniform Incident Distributions}

We now extend our analysis to continuous incident distributions, specifically considering two cases: a Uniform incident distribution, where incidents occur with equal probability across the unit square, $\eta = \text{Unif}([0,1] \times [0,1]),$
and a Mixture-of-Uniform incident distribution, where the unit square \([0,1] \times [0,1]\) is divided into four equal subsquares, each with a different probability weight: \begin{align*}\eta = \text{MixUnif}([0,1] \times [0,1]) :=& 0.1 \text{Unif}([0,0.5] \times [0,0.5]) + 0.2 \text{Unif}([0.5,1] \times [0,0.5]) \\ 
    &\quad + 0.3 \text{Unif}([0,0.5] \times [0.5,1]) + 0.4 \text{Unif}([0.5,1] \times [0.5,1]).\end{align*}
These distributions provide a more realistic representation of emergency incidents compared to discrete demand on few points. For instance, the model proposed in~\cite{van2024modeling} assumes that OHCA incidents follow a uniform distribution within each unit area, resulting in an overall mixture of uniform distributions across a city. While real-world applications involve larger and more complex spatial domains, our setup—focusing on a \( 1 \times 1 \) unit square—serves as a simplified yet insightful setting. Computing accurate numerical solutions using fc-FW allows us to analyze the fine-grained structure of optimal measures, offering insights into microscopic allocation patterns that would be difficult to derive analytically.

\begin{figure}[htb]
\FIGURE{
\parbox{\textwidth}{
\centering
    \subcaptionbox{Opt Volunteer Measure $\hat{\mu}^*$ ($b=1$).}{
        \includegraphics[width=0.3\textwidth]{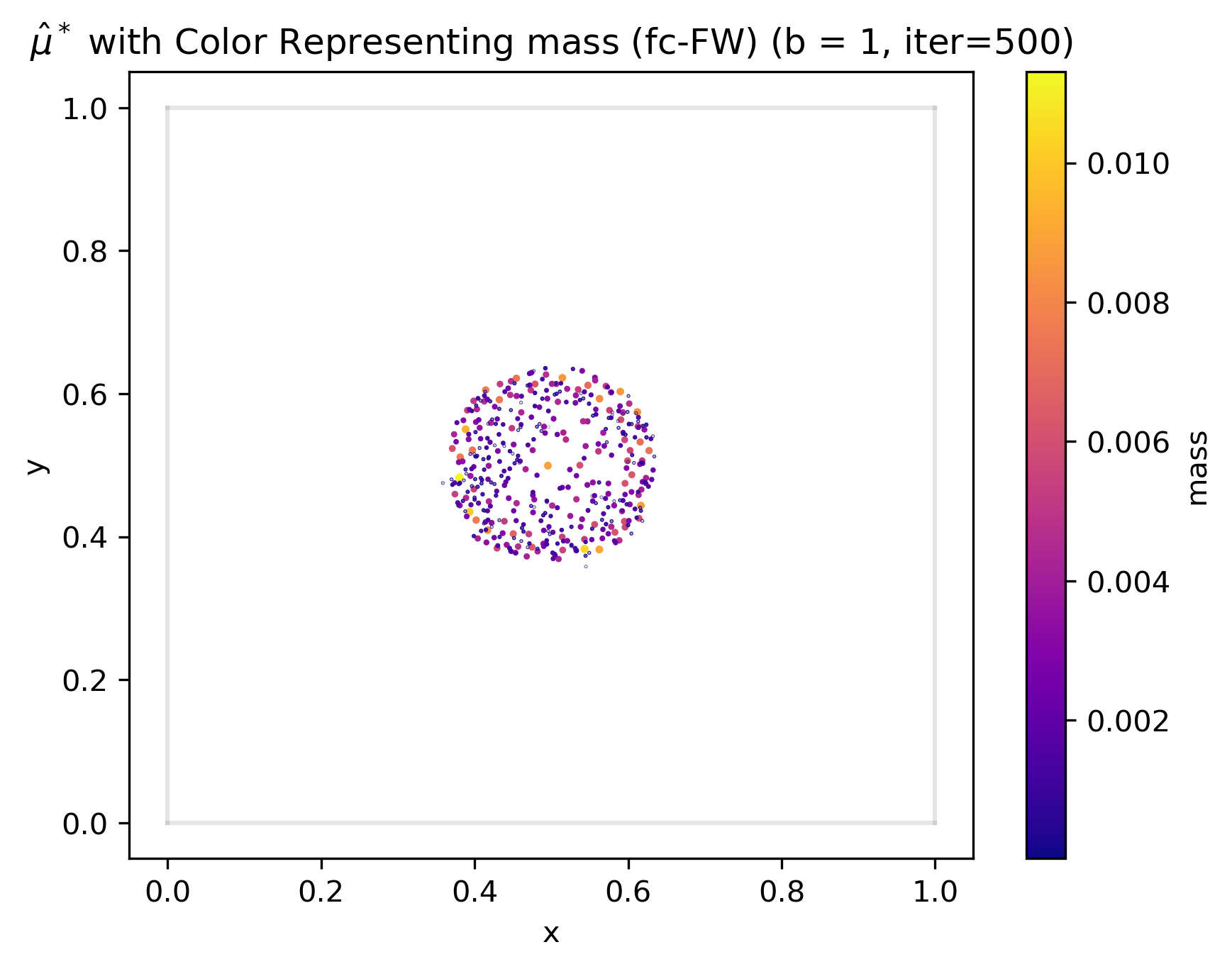}
    }
    \hfill
    \subcaptionbox{Opt Volunteer Measure $\hat{\mu}^*$ ($b=5$).}{
        \includegraphics[width=0.3\textwidth]{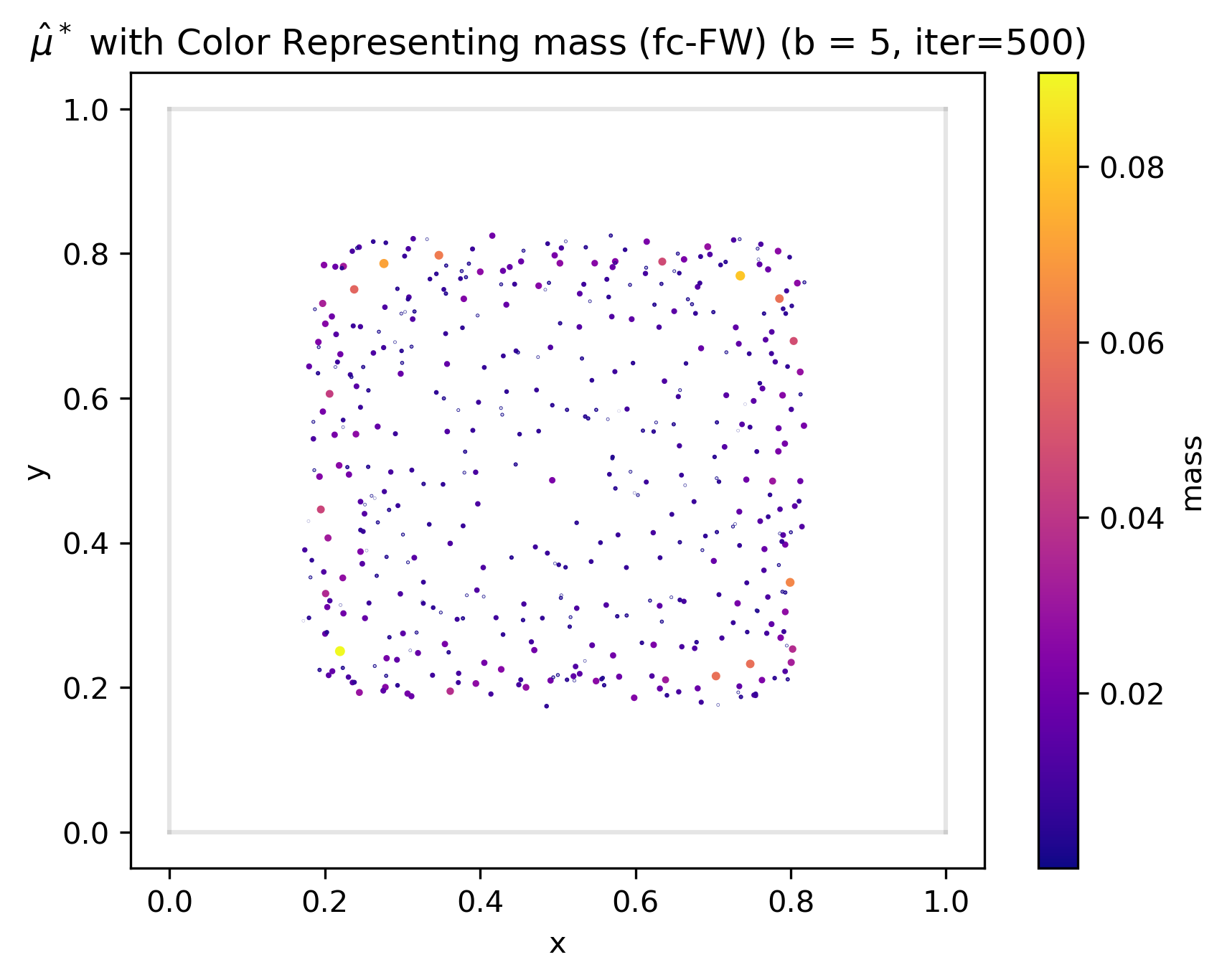}
    }
    \hfill
    \subcaptionbox{Opt Volunteer Measure $\hat{\mu}^*$ ($b=20$).}{
        \includegraphics[width=0.31\textwidth]{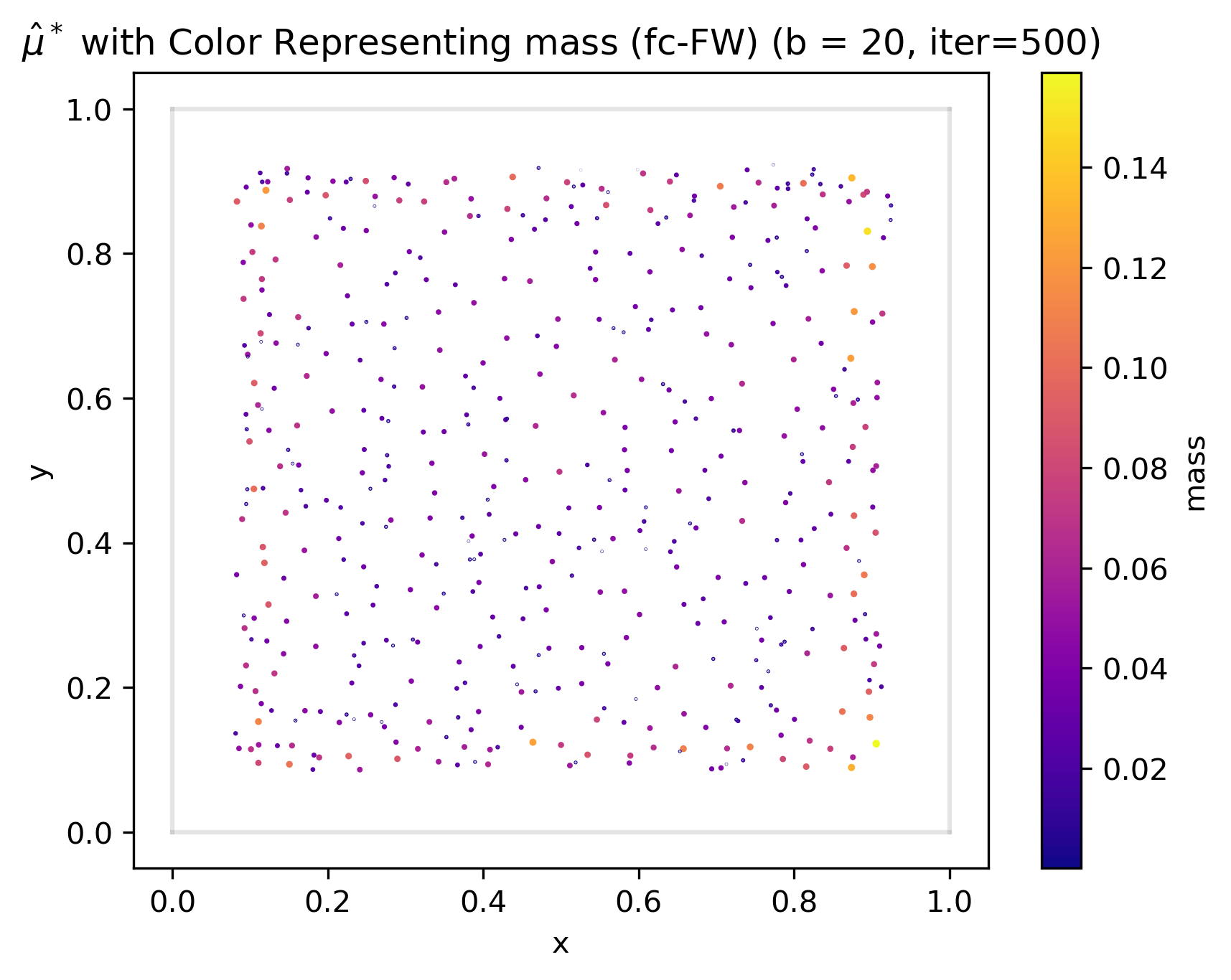}
    }
    \\[1ex]
    \subcaptionbox{Value of $J_n(\mu_k)$ ($b=1$, $n=1000$).}{
        \includegraphics[width=0.3\textwidth]{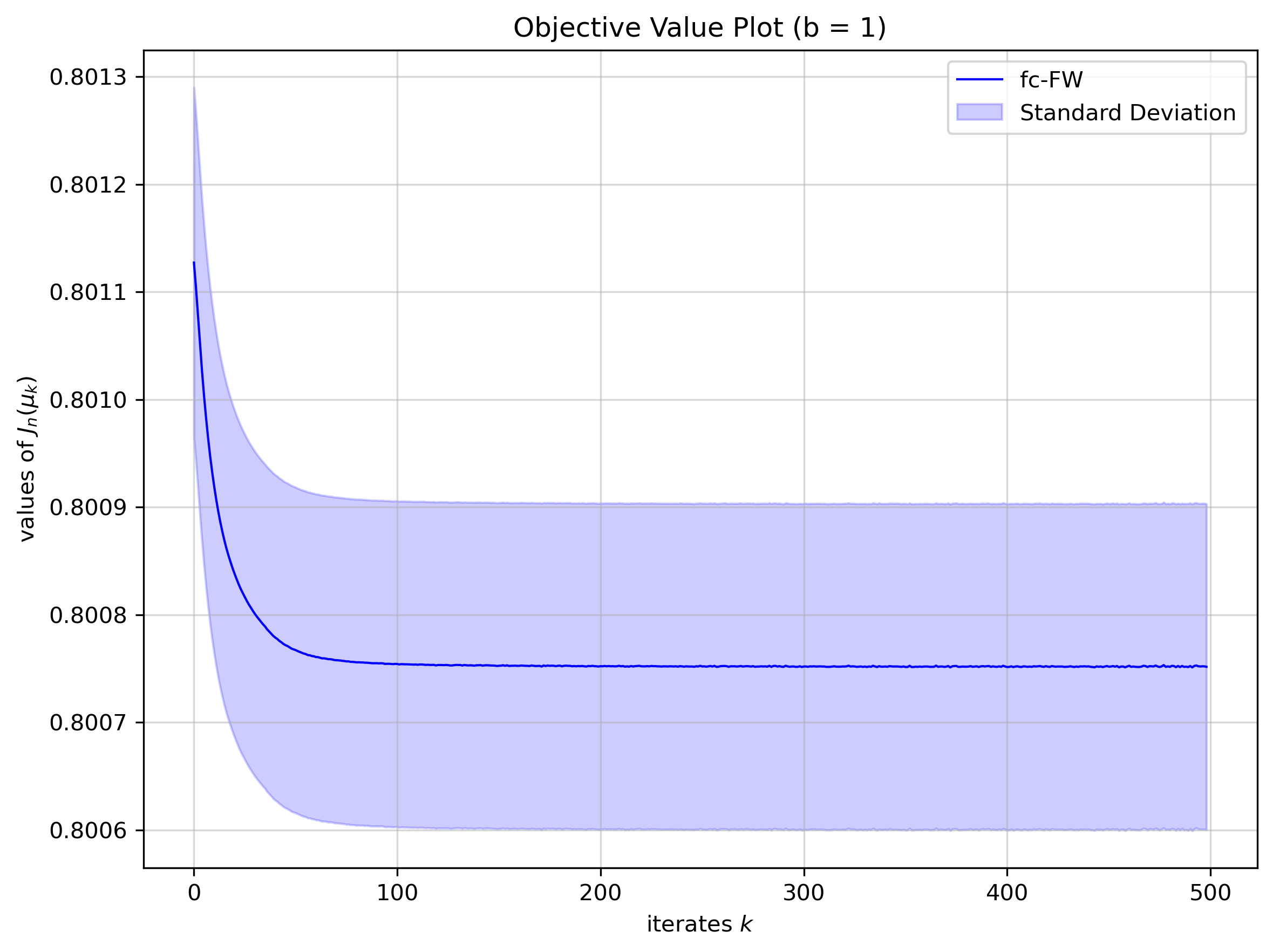}
    }
    \hfill
    \subcaptionbox{Value of $J_n(\mu_k)$ ($b=5$, $n=1000$).}{
        \includegraphics[width=0.3\textwidth]{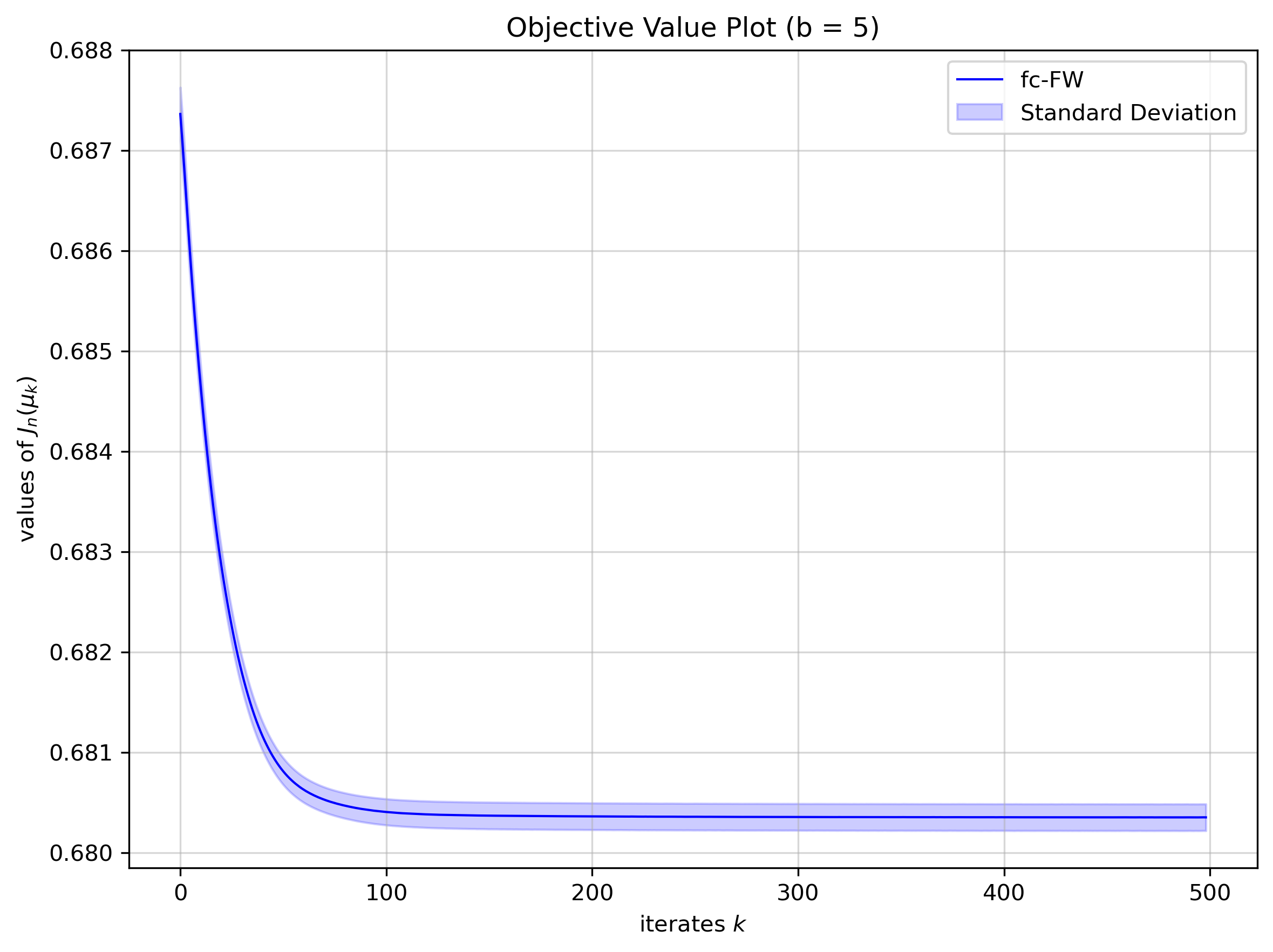}
    }
    \hfill
    \subcaptionbox{Value of $J_n(\mu_k)$ ($b=20$, $n=1000$).}{
        \includegraphics[width=0.3\textwidth]{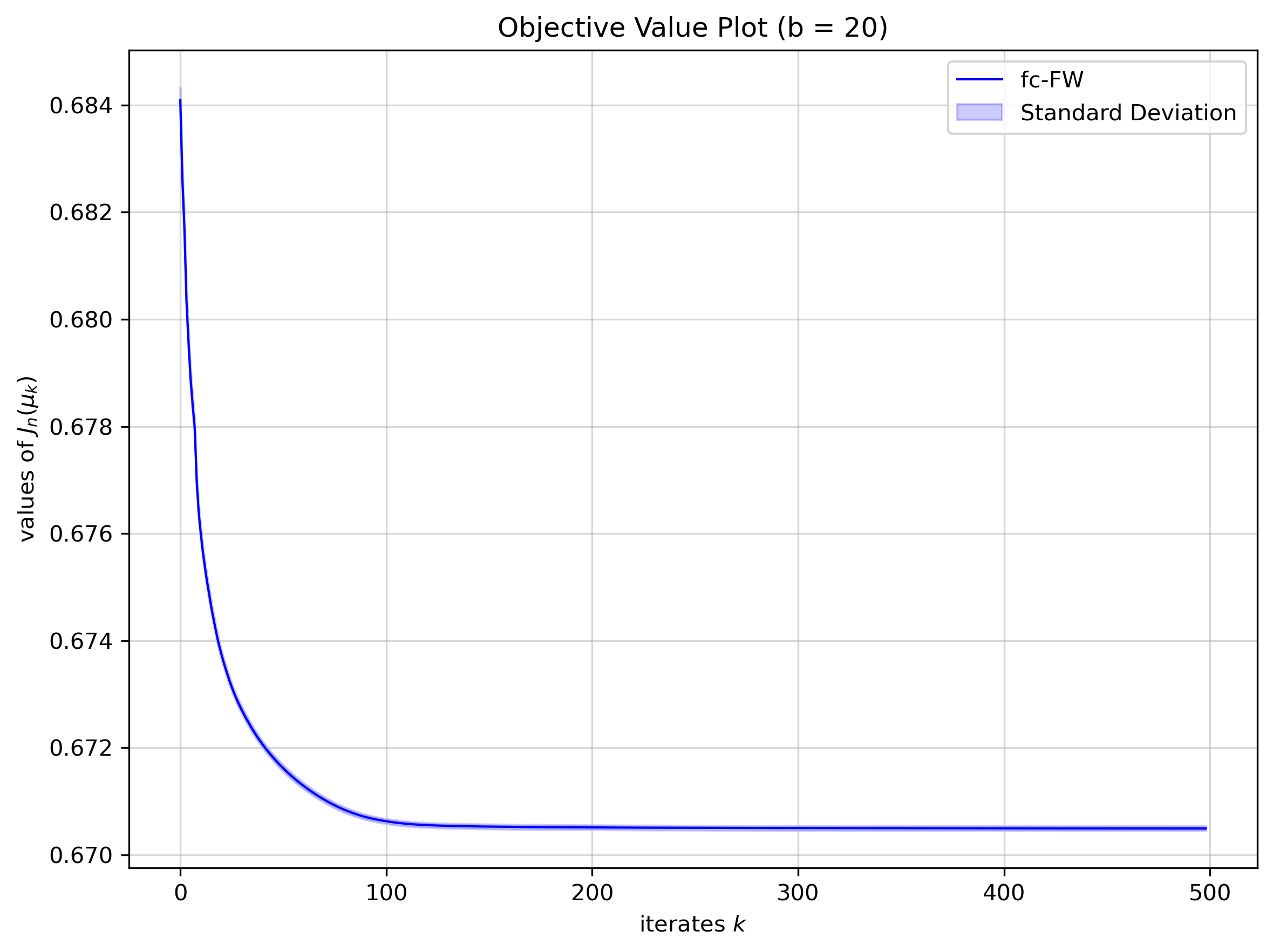}
    }
    }
}
{
Emergency Response with Uniform Incident Distribution
\label{fig:emergencyResUnif}}
{
    Figures (a), (b), and (c) show the optimal volunteer measures \( \hat{\mu}^* \) obtained using fc-FW (Algorithm~\ref{alg:fcFW}) after 500 iterations, with total volunteer mass set to \( b = 1 \), \( 5 \), and \( 20 \), respectively. The color intensity represents the assigned mass at each location, with lighter colors indicating higher density. Figures (d), (e), and (f) illustrate the convergence of \( J_n(\mu_k) \), a sample-average estimator of the objective function \( J(\mu) \), computed using \( n = 1000 \) sampled incident locations. Results are averaged over 20 independent trials, with shaded regions indicating $ \pm 1 $ standard deviation of \( J_n(\mu_k) \) values across these trials at each iteration.
}
\end{figure}

We first analyze a baseline scenario where emergency incidents follow a uniform distribution over the unit square, \( \eta = \text{Unif}([0,1] \times [0,1]) \). Figure~\ref{fig:emergencyResUnif} presents the optimal volunteer measures obtained using fc-FW after 500 iterations for total masses \( b = 1, 5, \) and \( 20 \). Despite the uniform distribution of incidents, the optimal volunteer allocation is not uniform. In Figure~\ref{fig:emergencyResUnif}(a), where \( b = 1 \), the measure is highly concentrated in the central region, suggesting that with limited resources, the best strategy is to allocate volunteers to a small but high-impact area rather than distributing them evenly. As shown in Figures~\ref{fig:emergencyResUnif}(b) and~\ref{fig:emergencyResUnif}(c), increasing \( b \) expands the support of the optimal measure, covering a larger portion of the region while maintaining a structured allocation pattern. Figures~\ref{fig:emergencyResUnif}(d)-(f) illustrate the convergence of fc-FW. 

\begin{figure}[htb]
\FIGURE{
\parbox{\textwidth}{
\centering
    \subcaptionbox{Opt Volunteer Measure $\hat{\mu}^*$ ($b=1$).}{
        \includegraphics[width=0.3\textwidth]{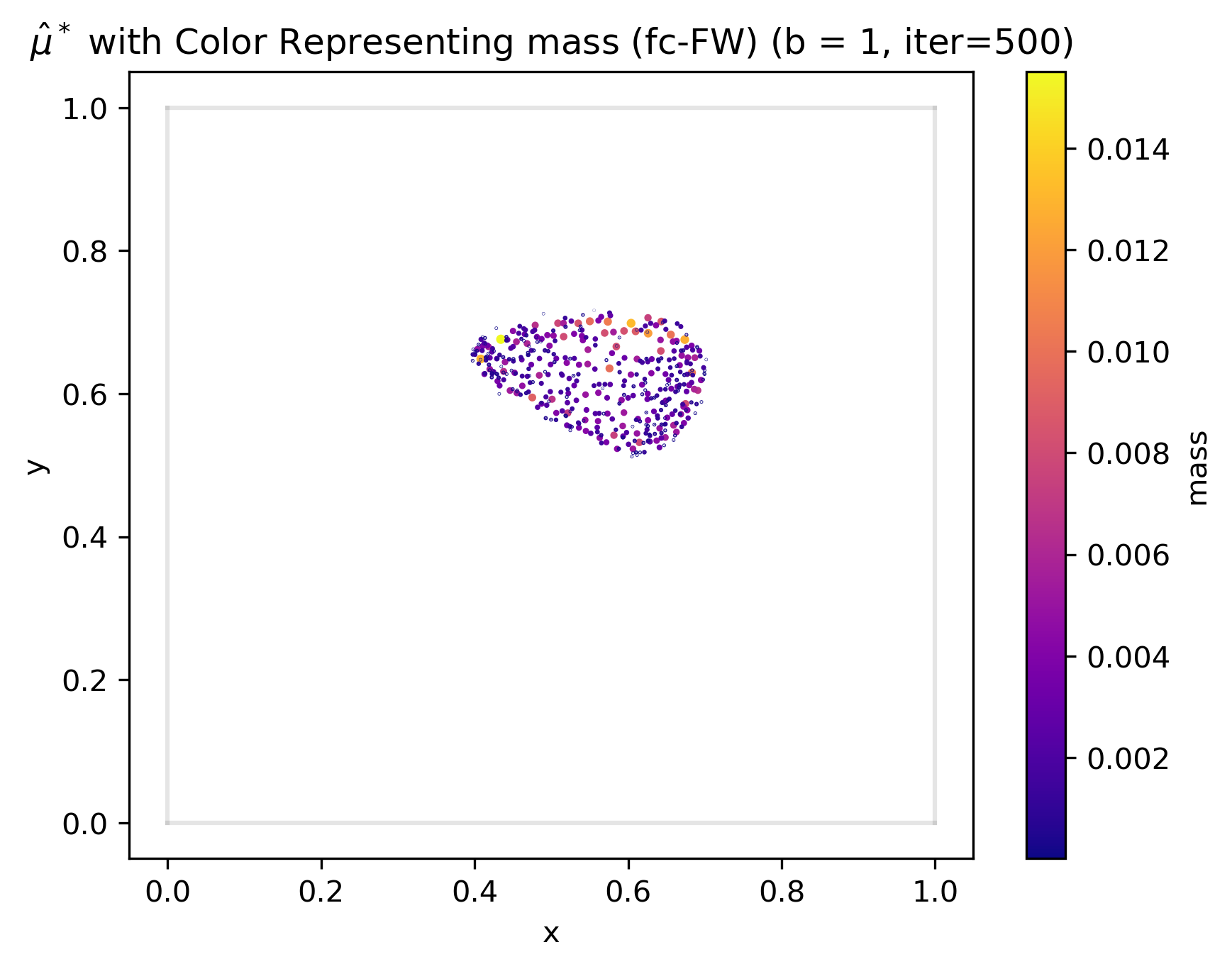}
    }
    \hfill
    \subcaptionbox{Opt Volunteer Measure $\hat{\mu}^*$ ($b=5$).}{
        \includegraphics[width=0.3\textwidth]{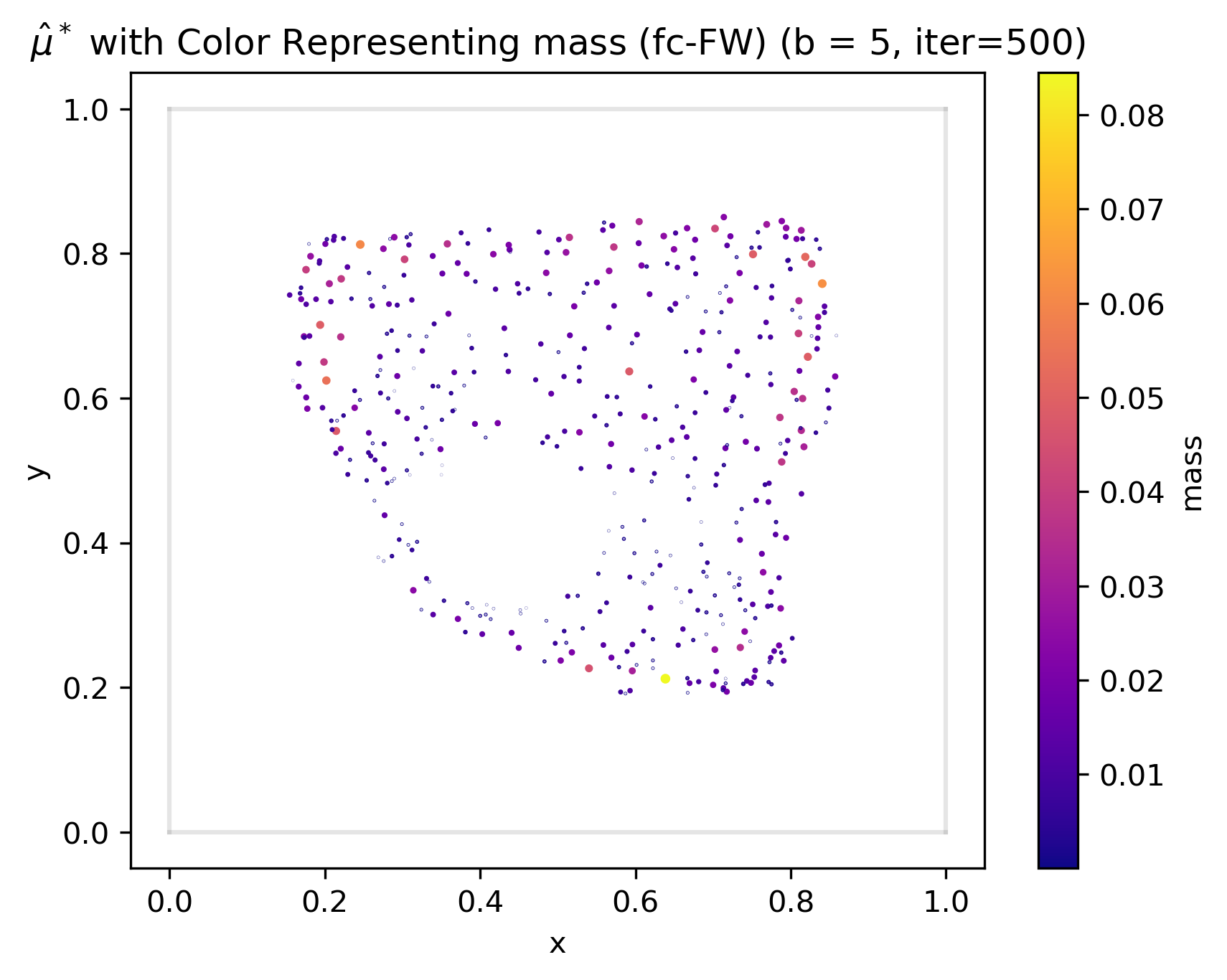}
    }
    \hfill
    \subcaptionbox{Opt Volunteer Measure $\hat{\mu}^*$ ($b=20$).}{
        \includegraphics[width=0.31\textwidth]{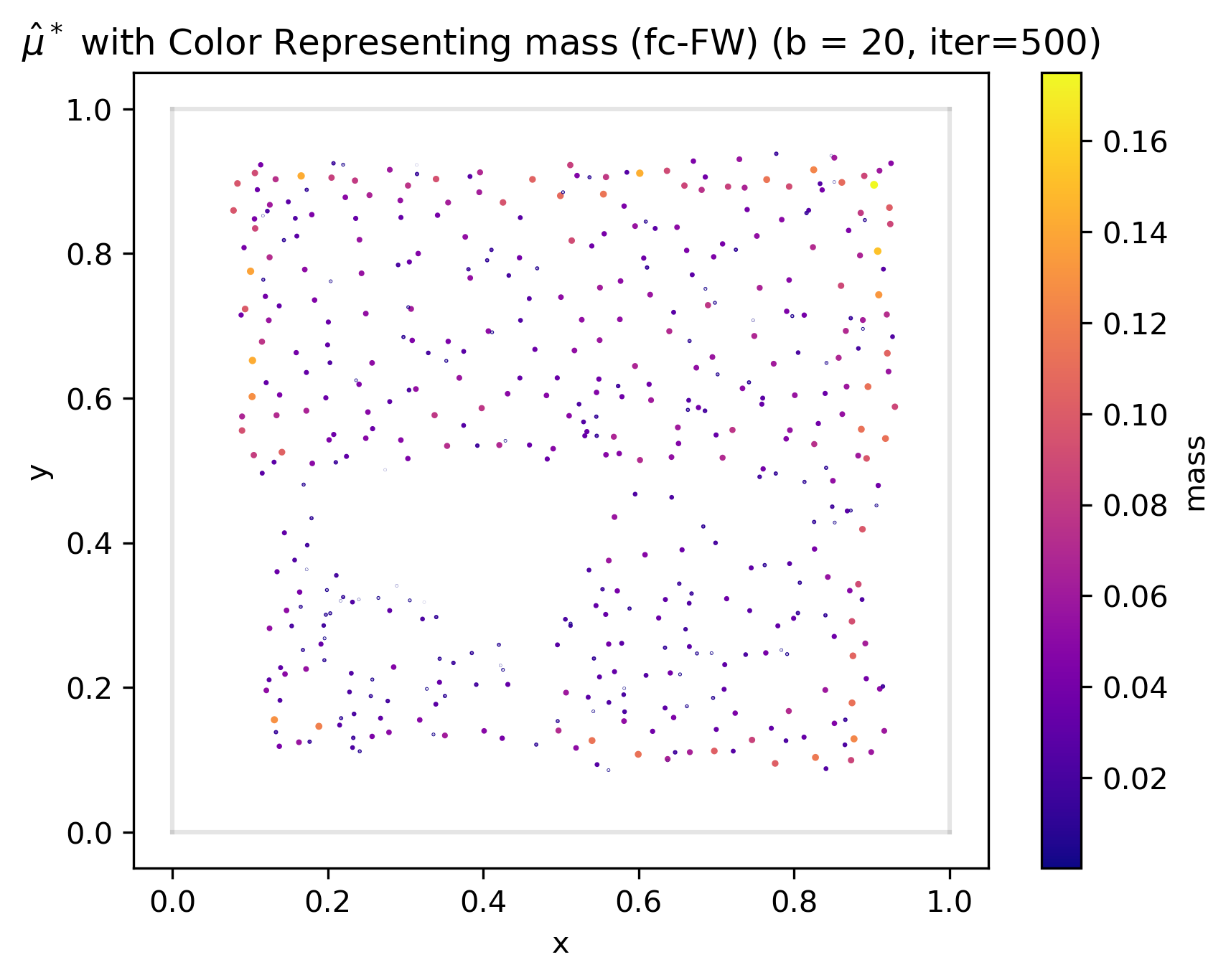}
    }
    \\[1ex]
    \subcaptionbox{Value of $J_n(\mu_k)$ ($b=1$, $n=1000$).}{
        \includegraphics[width=0.3\textwidth]{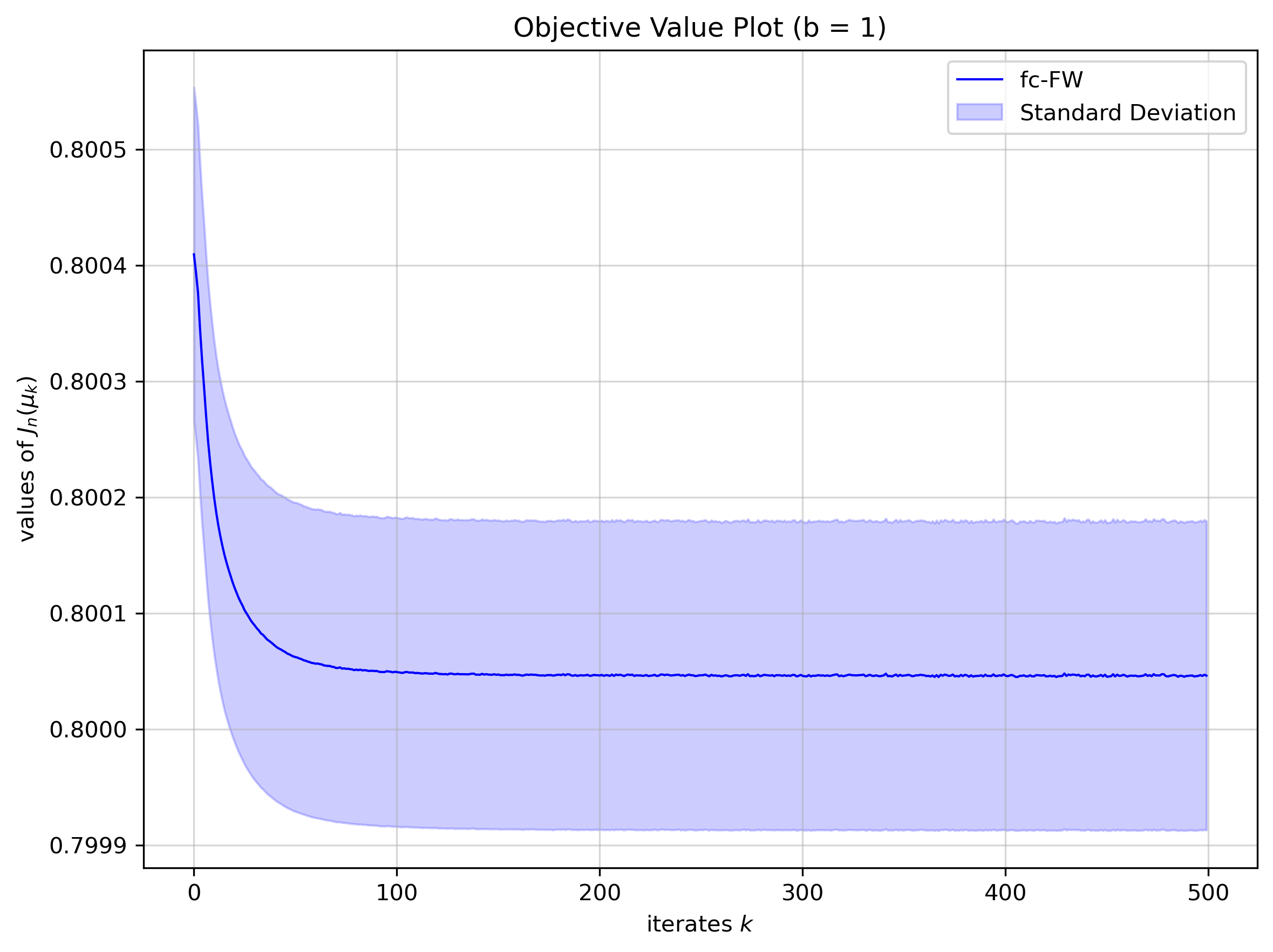}
    }
    \hfill
    \subcaptionbox{Value of $J_n(\mu_k)$ ($b=5$, $n=1000$).}{
        \includegraphics[width=0.3\textwidth]{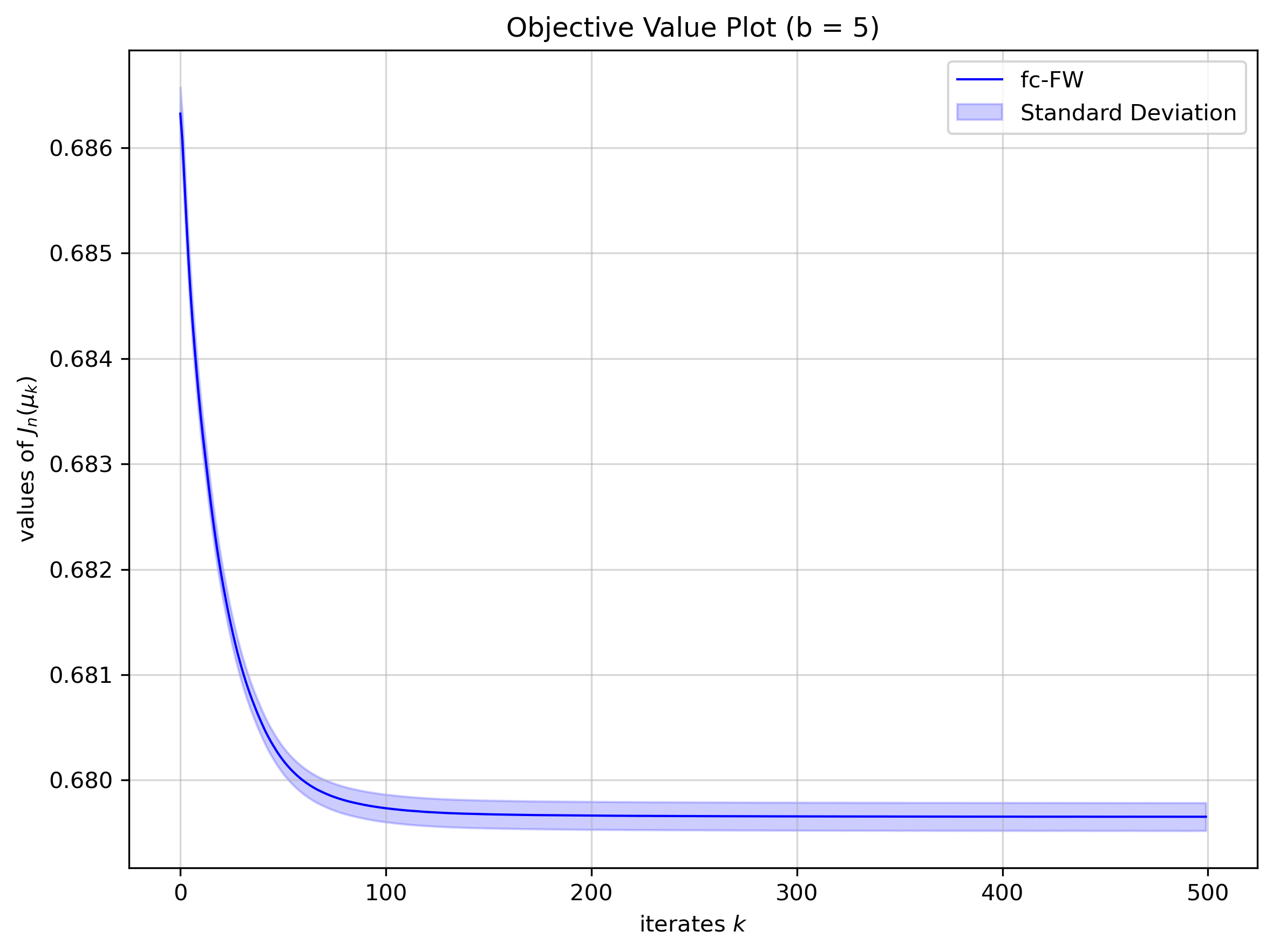}
    }
    \hfill
    \subcaptionbox{Value of $J_n(\mu_k)$ ($b=20$, $n=1000$).}{
        \includegraphics[width=0.3\textwidth]{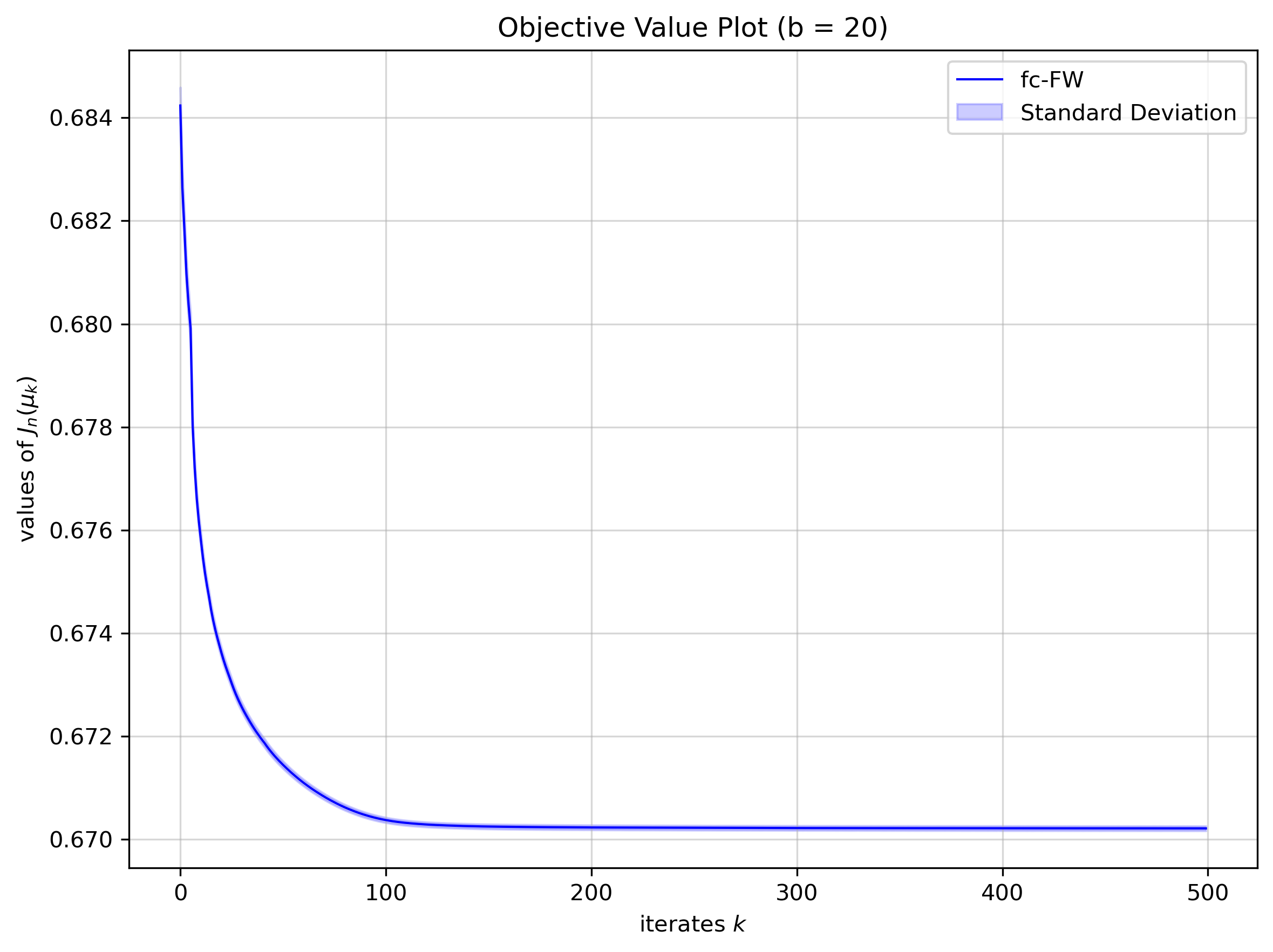}
    }
    }
}
{
Emergency Response with Mixture-of-Uniform Incident Distribution.
\label{fig:emergencyResMix}}
{
    Figures (a), (b), and (c) show the optimal volunteer measures \( \hat{\mu}^* \) obtained using fc-FW (Algorithm~\ref{alg:fcFW}) after 500 iterations for total volunteer masses \( b = 1, 5, \) and \( 20 \), respectively. The color intensity represents the assigned mass at each location, with lighter colors indicating higher density. Figures (d), (e), and (f) illustrate the convergence of \( J_n(\mu_k) \), a sample-average estimator of the objective function \( J(\mu) \), computed using \( n = 1000 \) sampled incident locations. Results are averaged over 20 independent trials, with shaded regions indicating $ \pm 1 $ standard deviation of \( J_n(\mu_k) \) values across these trials at each iteration.
}
\end{figure}

Building on the uniform case, we next examine a more structured incident distribution where emergency incidents follow a Mixture-of-Uniform distribution, \( \eta = \text{MixUnif}([0,1] \times [0,1]) \). Figure~\ref{fig:emergencyResMix} presents the optimal volunteer measures obtained using fc-FW for total masses \( b = 1, 5, \) and \( 20 \), while Figures~\ref{fig:emergencyResMix}(d)-(f) illustrate the convergence of fc-FW across different values of \( b \).

The results in Figures~\ref{fig:emergencyResMix}(a)-(c) show that although the incident distribution follows a mixture of uniform distributions, the optimal volunteer measure does not inherit this structure. Instead, the allocation pattern is skewed, with volunteers concentrating in high-probability regions rather than being proportionally distributed across all subregions. Moreover, some areas receive little to no allocation, indicating that the optimal strategy prioritizes select locations rather than simply mirroring the incident distribution.

As in the uniform case, the structure of the optimal measure is influenced by the total volunteer mass \( b \). When \( b = 1 \) (Figure~\ref{fig:emergencyResMix}(a)), the allocation is highly localized, prioritizing a small but high-impact region. As \( b \) increases to 5 and 20 (Figures~\ref{fig:emergencyResMix}(b)-(c)), the support of the optimal measure expands, covering a broader area while still reflecting the non-uniform incident probabilities. 

These findings provide valuable insights into the microscopic structure of optimal volunteer allocation within the unit square for the OHCA emergency response problem~\eqref{restrictedopt}. The optimal volunteer distribution does not directly mirror the incident distribution. With a limited budget, resources are concentrated in a small but high-impact region, whereas a larger budget enables broader spatial coverage. These results also demonstrate the effectiveness of the proposed fc-FW method in computing optimal allocations.

\subsection{Optimal Volunteer Allocation in Auckland City.}

\begin{figure}[htb]
\FIGURE
{\includegraphics[width=0.6\textwidth]{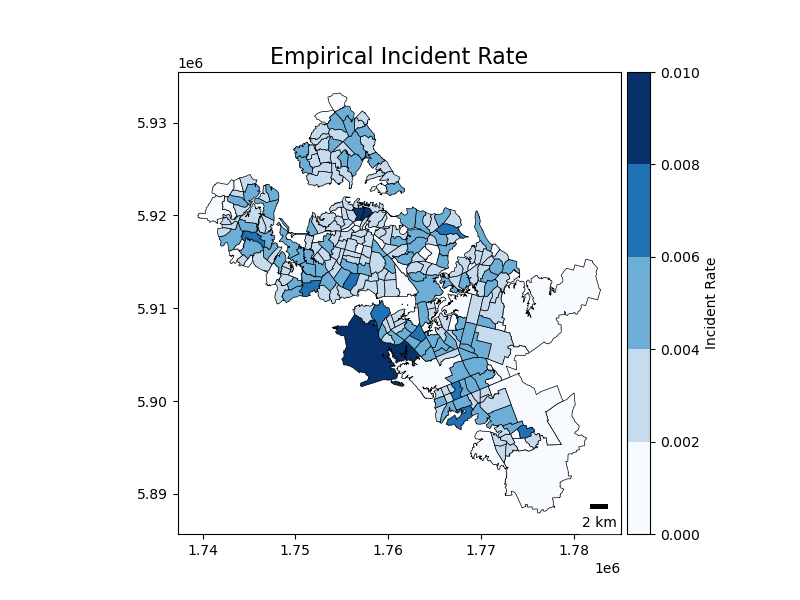} }
{The empirical incident rate per area unit in Auckland. \label{fig:empiricalRate} }
{The empirical incident rate per area unit in Auckland, New Zealand. This visualization follows the representation in~\cite{van2024modeling} but is reproduced here for clarity.}
\end{figure}

In this experiment, we consider a real-world scenario in Auckland, New Zealand, following many of the elements in the setup in van den Berg et al.~\cite{van2024modeling}. Our calculations omit a number of practical elements that are present in that study for simplicity, so our results should not be construed as providing realistic predictions. Rather, our goal is to demonstrate that our approach scales to city-scale systems.

The city is partitioned into a finite number of area units indexed by \( l \in \mathcal{L} = \{1, 2, \dots, L\} \), where \( L = 287 \). Each unit \( l \) has an associated incident probability \( \lambda_l \), and incidents are assumed to occur uniformly within each unit, denoted by the distribution \( u_l \). The resulting incident distribution is a mixture of uniform distributions:
\begin{equation*}
    \eta = \sum_{l \in \mathcal{L}} \lambda_l u_l, \quad \sum_{l \in \mathcal{L}} \lambda_l = 1.
\end{equation*}
The empirical incident rates \( \lambda_l \) are estimated in~\cite{van2024modeling} using the GoodSAM dataset, which contains real-world records of out-of-hospital cardiac arrests (OHCAs) in Auckland from 2013 to 2020. Since we do not have access to the original data, we directly use the published estimates without re-estimation. Figure~\ref{fig:empiricalRate} visualizes the incident rates across the area units.

\begin{figure}[htb]
\FIGURE{
\parbox{\textwidth}{
\centering
    \subcaptionbox{Optimal Volunteer Distribution ($b=50$).}{
        \includegraphics[width=0.48\textwidth]{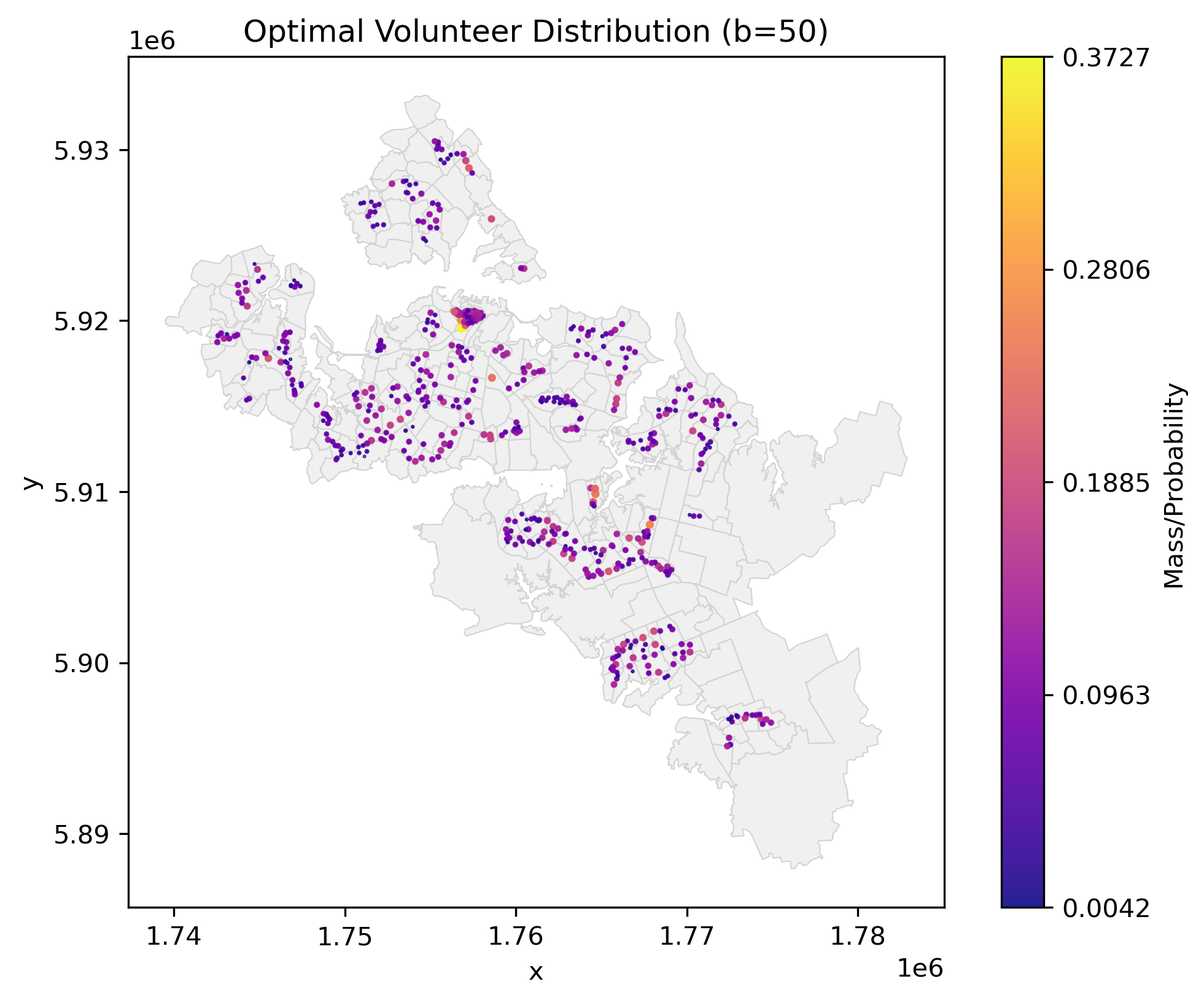}
    }
    \hfill
    \subcaptionbox{Optimal Volunteer Distribution ($b=500$).}{
        \includegraphics[width=0.48\textwidth]{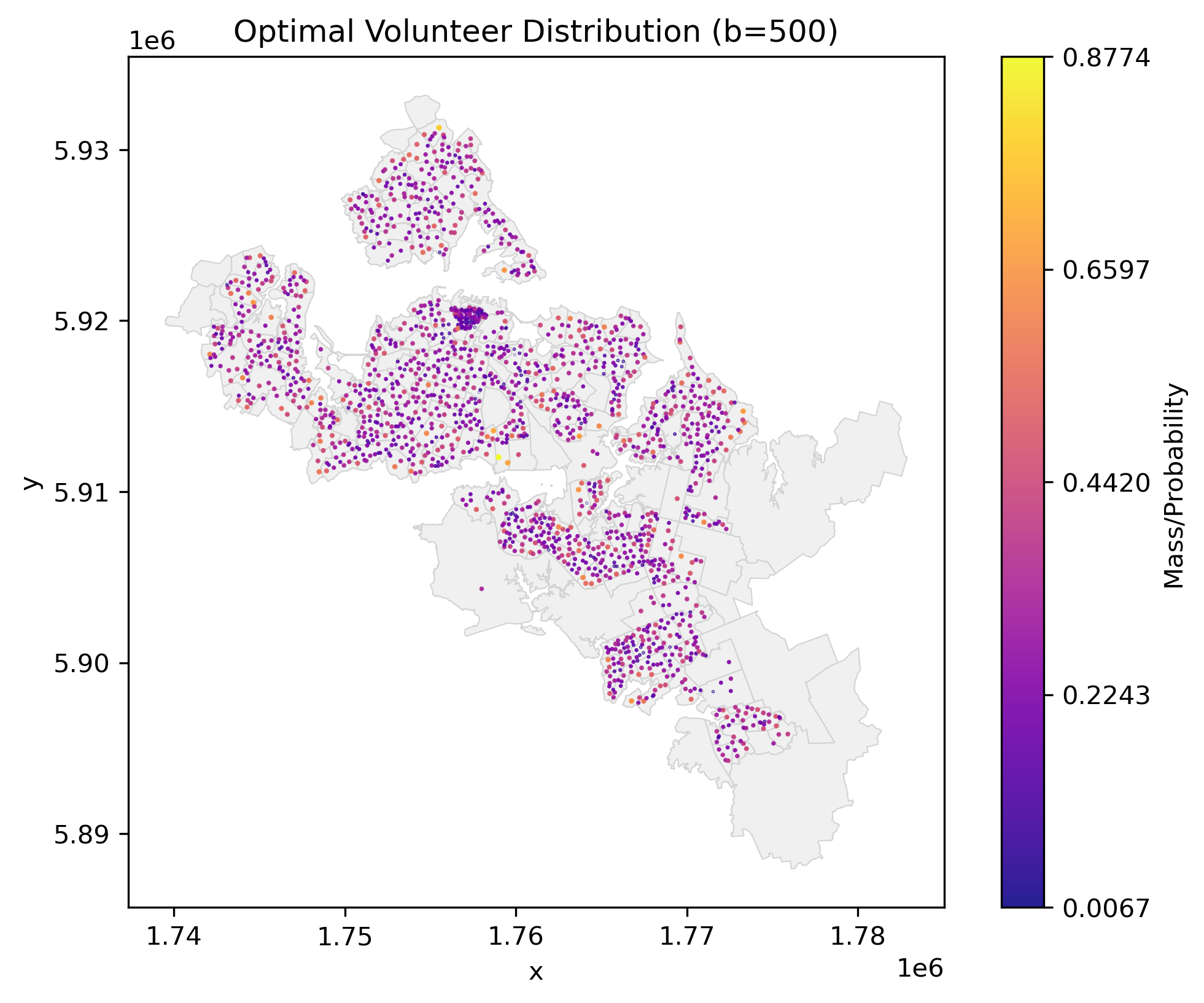}
    }
    \\[1ex]
    \subcaptionbox{Optimal Volunteer Distribution ($b=5000$).}{
        \includegraphics[width=0.48\textwidth]{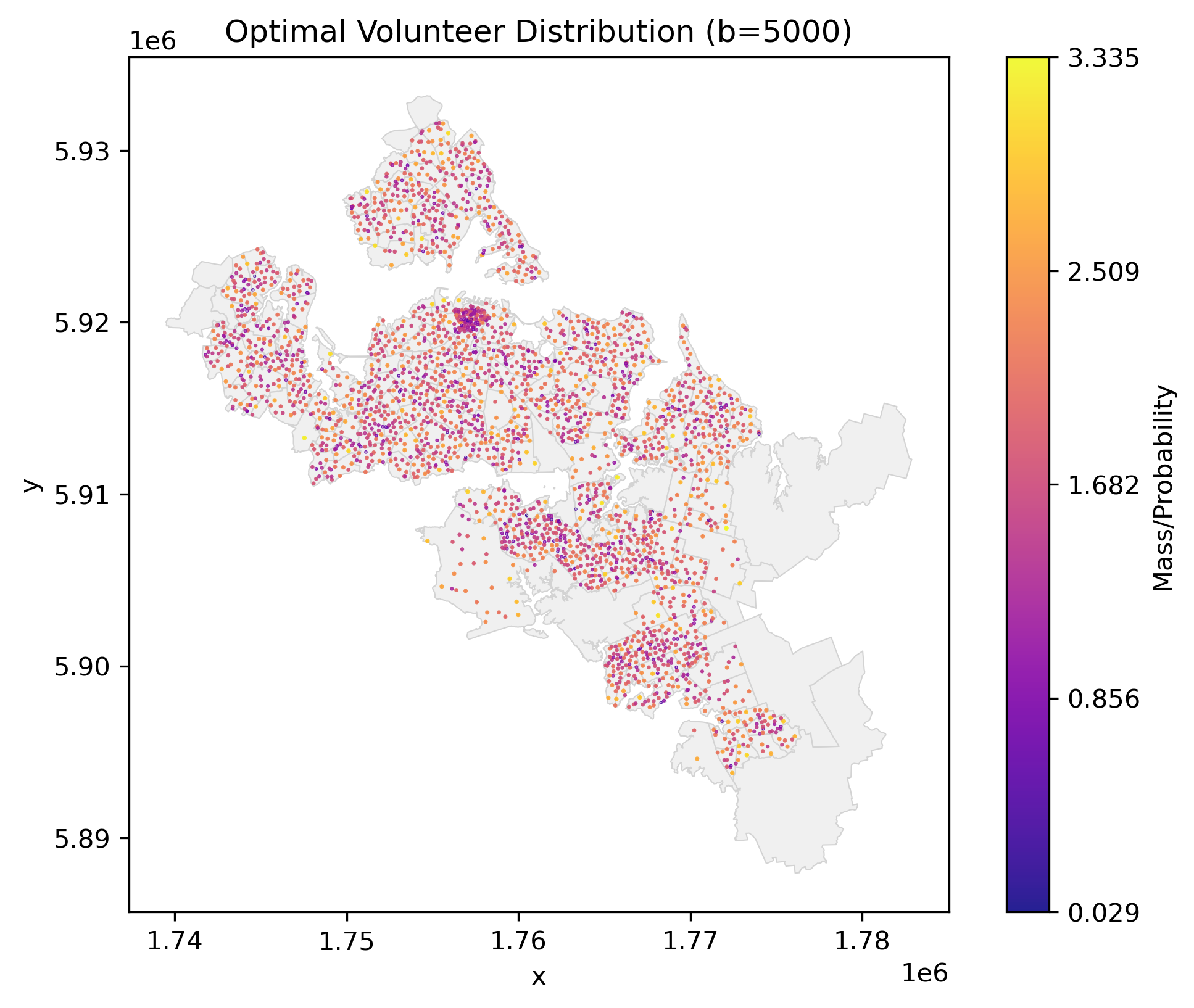}
    }
    \subcaptionbox{Convergence of \( J_n(\mu_k) \) for \( b = 50, 500, 5000 \).}{
        \includegraphics[width=0.48\textwidth]{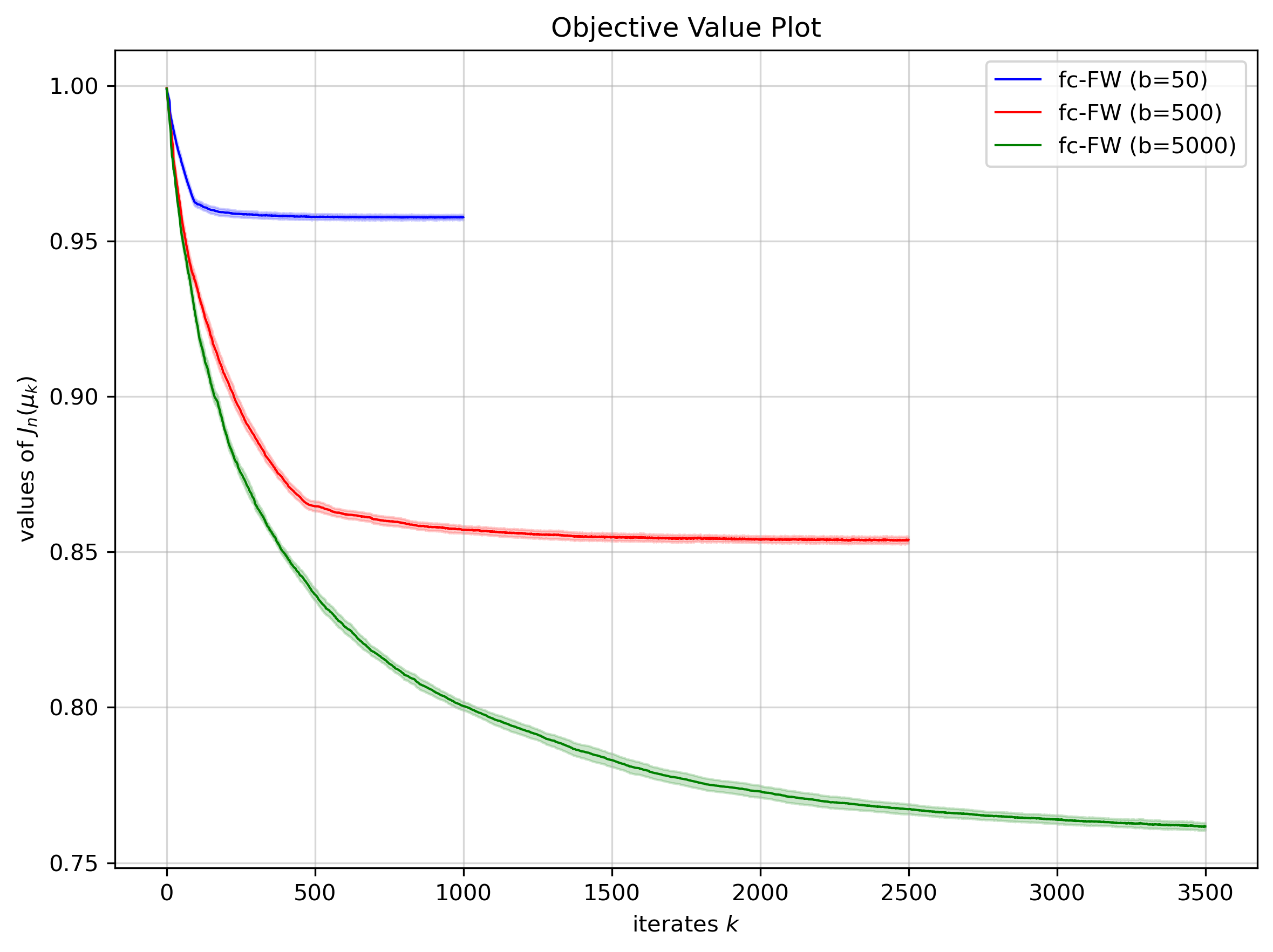}
    }
    }
}
{
Optimal Volunteer Allocation for Auckland City.
\label{fig:emergencyAuckland}}
{
     Figures (a), (b), and (c) show the optimal volunteer distributions in Auckland City obtained using fc-FW with total masses \( b = 50 \), \( 500 \), and \( 5000 \), respectively. The color intensity indicates the assigned mass at each location, with lighter colors representing higher volunteer density. The algorithm runs for 1000, 2500, and 3500 iterations for \( b = 50 \), \( 500 \), and \( 5000 \), respectively. Figure (d) presents the convergence of the sample-average estimator \( J_n(\mu_k) \) over iterations for each case, averaged over 10 independent trials with shaded regions indicating standard deviations across runs.
}
\end{figure}

We consider three total volunteer mass values: \( b = 50 \), \( 500 \), and \( 5000 \). The choices \( b = 500 \) and \( 5000 \) follow the setting in Section~6.3 of~\cite{van2024modeling}, where these values are used to reflect moderate and large-scale community responder deployments in Auckland. These numbers correspond to realistic participation rates based on comparisons to established CFR systems. We additionally include \( b = 50 \) to explore more constrained allocation patterns under limited volunteer resources.

Figure~\ref{fig:emergencyAuckland} shows the optimal volunteer allocations computed using fc-FW for each of these volunteer resource levels. Again, while the incident distribution \( \eta \) is a mixture of uniform distributions, the optimal volunteer measures do not directly reflect this structure. When \( b = 50 \), the volunteer mass is concentrated in a few high-risk areas, again reflecting a strategy of focusing limited resources on the most critical zones. As the total mass increases to \( b = 500 \) and \( 5000 \), the allocation becomes more spread out, covering a larger portion of the city. However, even with a large total mass such as \( b = 5000 \), the optimal distribution does not appear uniform; it remains concentrated in high-demand areas where volunteer impact is greatest. These results are encouraging in that they demonstrate the scalability and effectiveness of fc-FW in realistic, complex spatial domains.

\section{Does the Choice of Norm Matter?}\label{sec:normMatter}
In the emergency response problem defined in~\eqref{mainopt} and its restricted version~\eqref{restrictedopt}, the \(L_2\)-norm is the default choice for measuring the distance between incidents and volunteers. Specifically, the objective function is expressed as  
\begin{equation*}
    J(\mu) = \int_{\mathcal{S}} \eta(\mathrm{d}y) \int_0^\infty \exp\Big\{-\mu(\bar{B}(y,t))\Big\} \, \mathrm{d}\death(t),
\end{equation*}  
where the norm appearing in \(B(x, t) = \{y \in \mathbb{R}^2: \|y-x\| \leq t\}\) is the \(L_2\)-norm, defined as \(\|x\| = \sqrt{x_1^2 + x_2^2}\), for $x \in \mathbb{R}^2$. In this section, we ask if there are other norms, e.g., the $L_1$ norm given as \(\|x\|_1 = |x_1| + |x_2|\) for \(x \in \mathbb{R}^2\), that might be more suitable from the standpoint of computation, or from the standpoint of implementation in urban environments. 

The first insight corresponding to the use of alternative norms has to do with the corresponding ordering of the objective functions. Specifically, notice that when using the $L_1$ norm, the objective function becomes
\begin{equation}\label{eq:objL1}
    J_1(\mu) := \int_{\mathcal{S}} \eta(\mathrm{d}y) \int_0^\infty \exp\Big\{-\mu(\bar{B}_1(y,t))\Big\} \, \mathrm{d}\death(t),
\end{equation} 
where \(B_1(x, t) := \{y \in \mathbb{R}^2: \|y-x\|_1 \leq t\}\) representing the $L_1$ ball centered at \(x\) and having radius \(t\). Moreover, the $L_1$ ball lies entirely within the $L_2$ ball of the same center and radius. This implies that for any volunteer measure, \(\mu(\bar{B}(y,t)) \geq \mu(\bar{B}_1(y,t))\), and  consequently,
\begin{equation}\label{eq:L1bound}
    J(\mu) \leq J_1(\mu) \quad \forall \mu \in \mathcal{M}^+(\mathrm{cvx}(\mathcal{S}), b),
\end{equation}
indicating that \(J_1(\mu)\) is an upper bound for \(J(\mu)\). (Analogous orderings for other norms follow in a straightforward manner.) Apart from providing the upper bound in~\eqref{eq:L1bound}, the objective $J_1$ in~\eqref{eq:objL1} is also interesting in its own right since it better reflects emergency response settings in many grid-based urban layouts, while also providing computational advantages, as we shall see, below.

To investigate this further, we revisit the special case of \(n\) fixed demand points discussed in Section~\ref{sec:specialCases}, but with $J_1$ in~\eqref{eq:objL1} as the objective. Specifically, assume that the OHCA events occur at \(n\) fixed demand points \(\mathcal{S} = \{y_1, y_2, \dots, y_n \}\), with the incident measure defined as: $\eta = \sum_{i=1}^n \lambda_i \delta_{y_i}$, $\sum_{i=1}^n \lambda_i = 1$, $\lambda_1, \dots, \lambda_n \geq 0.$ The corresponding emergency response problem then becomes
\begin{align}\label{optL1}
    \text{minimize} & \quad J_1(\mu) = \sum_{i=1}^n \lambda_i \int_0^\infty \exp\Big\{-\mu(\bar{B}_1(y_i,t))\Big\} \, \mathop{d\death(t)}, \nonumber \\
    \text{subject to} & \quad \mu \in \mathcal{M}^+(\mathrm{cvx}(\mathcal{S}), b), \tag{$P(L_1)$}
\end{align}
where \(\bar{B}_1(y_i, t) = \{x \in \mathbb{R}^2 : \|x - y_i\|_1 \leq t\}\) represents the $L_1$ ball centered at \(y_i\) with radius \(t\). The influence function for the objective $J_1$ is
\begin{align}\label{eq:influenceL1}
    h_{\mu}(x) &= \sum_{i=1}^n \lambda_i \int_0^\infty 
    \left( \mu\left(\bar{B}_1(y_i,t) \right) - b \, \mathbb{I}_{[\|x-y_i\|_1, \infty)}(t) \right) 
    \exp\left\{-\mu\left(\bar{B}_1(y_i,t) \right)\right\} \, \mathop{d\death(t)} \nonumber \\ &= \sum_{i=1}^n \int_0^{\|x-y_i \|_1}  b\,\lambda_i\exp\left\{-\mu\left(\bar{B_1}(y_i,t) \right) \right\} \, \mathop{d\death(t)} + c,
\end{align} where $c$ is a constant in $\mathbb{R}$.  

\begin{figure}[htbp]
    \centering
    \begin{tikzpicture}
    \coordinate (y1) at (2, 4); 
    \coordinate (y2) at (3, 2); 
    \coordinate (y3) at (5, 3); 
    \coordinate (y4) at (4, 1); 
    \draw[dashed] (0, 4) -- (6, 4); 
    \draw[dashed] (0, 2) -- (6, 2); 
    \draw[dashed] (0, 3) -- (6, 3); 
    \draw[dashed] (0, 1) -- (6, 1); 
    \draw[dashed] (2, 0) -- (2, 5); 
    \draw[dashed] (3, 0) -- (3, 5); 
    \draw[dashed] (5, 0) -- (5, 5); 
    \draw[dashed] (4, 0) -- (4, 5); 
        \fill[blue!20] (3, 2) rectangle (4, 3); 
    \filldraw[black] (y1) circle (2pt) node[above right] {\(y_1\)};
    \filldraw[black] (y2) circle (2pt) node[above left] {\(y_2\)};
    \filldraw[black] (y3) circle (2pt) node[above right] {\(y_3\)};
    \filldraw[black] (y4) circle (2pt) node[above right] {\(y_4\)};
    \node at (2, 5) [above] {\(a_{(1)}\)};
    \node at (3, 5) [above] {\(a_{(2)}\)};
    \node at (4, 5) [above] {\(a_{(3)}\)};
    \node at (5, 5) [above] {\(a_{(4)}\)};
    \node at (0, 4) [left] {\(b_{(4)}\)};
    \node at (0, 3) [left] {\(b_{(3)}\)};
    \node at (0, 2) [left] {\(b_{(2)}\)};
    \node at (0, 1) [left] {\(b_{(1)}\)};
    \node at (3.5, 2.5) [black] {\(S_{2,2}\)};
\end{tikzpicture}
    \caption{Grid for Demand Points (\(n=4\)).}
    \label{fig:GridL1}
\end{figure}
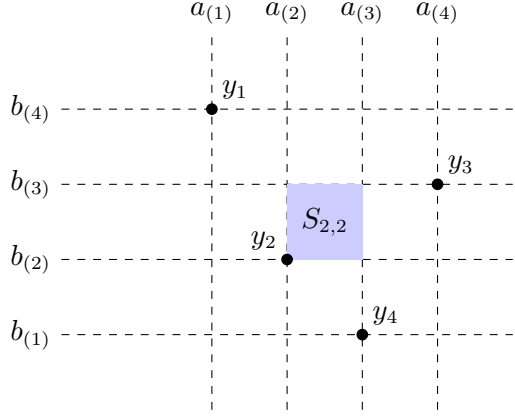

The influence function in~\eqref{eq:influenceL1} exhibits remarkable regularity that can be exploited to great effect. To see this, let \((a_i, b_i)\) for \(i = 1, \dots, n\) denote the coordinates of the demand locations \(y_i\). Define \(a_{(1)} \leq a_{(2)} \leq \dots \leq a_{(n)}\) and \(b_{(1)} \leq b_{(2)} \leq \dots \leq b_{(n)}\) as the order statistics of \(\{a_1, a_2, \dots, a_n\}\) and \(\{b_1, b_2, \dots, b_n\}\), respectively. Figure~\ref{fig:GridL1} illustrates the grid generated by \(y_1, \dots, y_n\), where each rectangle is defined as:
\begin{equation}\label{eq:rectangle}
    S_{j,k} := \{x = (a, b) \in \mathbb{R}^2 : a \in [a_{(j)}, a_{(j+1)}], \, b \in [b_{(k)}, b_{(k+1)}]\},
\end{equation}
for \(j, k = 1, \dots, n-1\). With this setup, the following lemma asserts that the influence function \(h_{\mu}(x)\) in~\eqref{eq:influenceL1} is piecewise strictly concave, that is, strictly concave over each rectangle.

\begin{lemma}[Strict Concavity]\label{lemma:piecewiseConcave}
    The influence function \(h_{\mu}(x)\) at any $\mu\in\mathcal{M}^+(\mathrm{cvx}(\mathcal{S}), b)$ in~\eqref{eq:influenceL1} is strictly concave within each rectangle \(S_{j,k}\) for \(j, k = 1, \dots, n-1\), that is, for any $t \in (0, 1)$ and \(x_1, x_2 \in S_{j,k}\) with $x_1 \ne x_2$,
    \begin{equation}
        h_{\mu}(t x_1 + (1-t)x_2) > t h_{\mu}(x_1) + (1-t)h_{\mu}(x_2).
    \end{equation}
\end{lemma}
\begin{proof}{Proof.} 
Since \( x_1, x_2 \in S_{j,k} \), for all \( y_i \) with \( i=1,\dots,n \), the vectors \( x_1 - y_i \) and \( x_2 - y_i \) have the same sign patterns. This implies that for all \( i=1,\dots,n \) and \( 0 \leq t \leq 1 \), we have
\[
\| t x_1 + (1 - t) x_2 - y_i \|_1 = t \| x_1 - y_i \|_1 + (1 - t) \| x_2 - y_i \|_1.
\]
Define $l_{i,1} := \| x_1 - y_i \|_1$, $l_{i,2} := \| x_2 - y_i \|_1$, and $m_i := t l_{i,1} + (1 - t) l_{i,2}$. After some computation we obtain
\begin{align*}
    &h_{\mu}(t x_1 + (1 - t) x_2) - t h_{\mu}(x_1) - (1 - t) h_{\mu}(x_2) \\
    &\qquad = b\sum_{i=1}^{n} \lambda_i \underbrace{\left( t\int_{l_{i,1}}^{m_i} \exp(-\mu(\bar{B}_1(y_i,t))) d\death(t)  - (1 - t) \int_{m_i}^{l_{i,2}} \exp(-\mu(\bar{B}_1(y_i,t))) d\death(t) \right)}_{I_i} \\ & \qquad=: b\sum_{i=1}^{n} \lambda_i I_i.
\end{align*}
Assuming without loss of generality that \( l_{i,1} < l_{i,2} \), and applying the mean value theorem, we have
\begin{align*}
    I_i = t \exp(-\mu(\bar{B}_1(y_i, c_1))) (\death(m_i) - \death(l_{i,1})) - t \exp(-\mu(\bar{B}_1(y_i, c_2))) (\death(l_{i,2}) - \death(m_i)),
\end{align*}
where \( c_1 \in (l_{i,1}, m_i) \) and \( c_2 \in (m_i, l_{i,2}) \). Since \( \death \) is strictly concave, it follows that $\death(m_i) > t \death(l_{i,1}) + (1 - t) \death(l_{i,2})$. Rearranging terms, we obtain
\[
I_i > t(1 - t)(\death(l_{i,2}) - \death(l_{i,1})) \left( \exp(-\mu(\bar{B}_1(y_i, c_1))) - \exp(-\mu(\bar{B}_1(y_i, c_2))) \right).
\]
Since \( \exp(-\mu(\bar{B}_1(y_i,t))) \) is non-increasing with respect to \( t \), \( c_1 < c_2 \), and \( \death \) is an increasing function, it follows that \( I_i > 0 \), and so 
$h_{\mu}(t x_1 + (1 - t) x_2) > t h_{\mu}(x_1) + (1 - t) h_{\mu}(x_2),$
as required. \hfill \Halmos
\end{proof}

Lemma~\ref{lemma:piecewiseConcave} is important for computation. Since a strictly concave function achieves its minimum at the extreme points of a convex set on which it is defined, the minimizers of \( h_{\mu}(x) \) in~\eqref{eq:influenceL1} must lie among the vertices of \( S_{j,k} \), giving the following result. 
\begin{lemma}\label{thm:vertex_min}
    For each rectangle \( S_{j,k} \), the minimizers of \( h_{\mu}(x) \) satisfy
    \begin{equation}\label{eq:optInfluence}
        \argmin_{x \in S_{j,k}} h_{\mu}(x) \subseteq V_{j,k}, \quad \forall \mu\in\mathcal{M}^+(\mathrm{cvx}(\mathcal{S}), b),
    \end{equation}
    where \( V_{j,k} \) denotes the set of vertices of \( S_{j,k} \):
    \[
    V_{j,k} = \{(a_{(j)}, b_{(k)}), (a_{(j+1)}, b_{(k)}), (a_{(j)}, b_{(k+1)}), (a_{(j+1)}, b_{(k+1)})\}.
    \]
\end{lemma}

By Lemma~\ref{thm:vertex_min}, for any $\mu\in\mathcal{M}^+(\mathrm{cvx}(\mathcal{S}), b)$ the minimizers of \( h_{\mu}(x) \) within each rectangle \( S_{j,k} \) lie at its vertices. Define the set of all such vertices, which correspond to the grid points, as  
\begin{equation}\label{eq:defV}
    V := \{ (a_{(j)}, b_{(k)}) \mid j, k = 1, \dots, n \}.
\end{equation}
Applying the optimality condition in Theorem~\ref{thm:optCondition}, we arrive at the following result.
\begin{theorem}\label{thm:supportSubset}
    The support of an optimal solution, \( \mu^* \), to problem~\eqref{optL1} is contained in the set of grid vertices, that is,
    \begin{equation}\label{eq:optSuppL1}
        \mathrm{supp}(\mu^*) \subseteq V.
    \end{equation}
\end{theorem}
\begin{proof}{Proof.}
Since \( V \) defined in~\eqref{eq:defV} is the set of all vertices, Theorem~\ref{thm:vertex_min} directly implies that the minimizers of the influence function at the optimal measure \( \mu^* \) satisfy
\begin{equation*}
    \argmin_{x\in\mathrm{cvx}(\mathcal{S})} h_{\mu^*}(x) \subseteq V.
\end{equation*}
Moreover, by the optimality condition in Theorem~\ref{thm:optCondition}, we have \( h_{\mu^*}(x) \geq 0 \) for all \( x \in \mathrm{cvx}(\mathcal{S}) \), with the minimum value attained at zero. Applying Lemma~\ref{lem:supportopt}, it follows that
\begin{equation*}
    \mathrm{supp}(\mu^*) \subseteq \left\{ x \in \mathrm{cvx}(\mathcal{S}) \mid h_{\mu^*}(x) = 0 \right\} = \argmin_{x \in \mathrm{cvx}(\mathcal{S})} h_{\mu^*}(x).
\end{equation*} Thus, equation~\eqref{eq:optSuppL1} follows. \hfill \Halmos
\end{proof}

Theorems~\ref{thm:vertex_min} and \ref{thm:supportSubset} are especially interesting because they reduce the search for the optimal support of $\mu^*$ to a finite candidate set $V$. The size of $V$ is $O(n^2)$, where $n$ is the cardinality of support of the demand random variable $\eta$. (The exponent $2$ on $n$ arises from the fact that the OHCA emergency response problem is stated in $\mathbb{R}^2$.) This suggests a simple ``look-up'' procedure for solving~\eqref{optL1} whereby the support of $\mu^*$ is deduced by observing $h_{\mu}$ at each of the points in $V$. 
A corresponding look-up procedure for more general demands $\eta$ is also evident --- sample from $\eta$ first and form the grid, and then implement the look-up procedure to estimate the support of the solution $\mu^*$.   

\section{Concluding Remarks}\label{sec:concluding}
We have developed and explored the use of a fully corrective Frank-Wolfe (fc-FW) algorithm for optimizing the volunteer measure in CFR systems that aim to improve survival rates from OHCA. We derived results that enhance the solution process, e.g., by establishing structural properties of the objective function including its von Mises differentiability, and we obtained an expression for the influence function. We developed additional convergence guarantees for the fc-FW algorithm that apply even when not globally minimizing the influence function in each iteration.

Study of stylized problems indicates that solutions are not easily identified in advance. Rather, they have complex structure that is plausible upon viewing the solutions, yet far from completely natural. Certainly, they do not simply mirror the OHCA probability distribution.

We demonstrate that computation is possible at city scale through computations for the city of Auckland, New Zealand.

Finally, we showed that if volunteers travel according to the Manhattan metric, and if the demand distribution is discrete, then the optimal volunteer measure concentrates on a grid. This observation should permit the development of specialized algorithms in such cases.

Future research within the sphere of emergency services might explore the impact on optimal volunteer measures of complex volunteer dispatch policies that require more than the closest volunteer to respond. Volunteer response is also being considered for emergencies beyond OHCA. How, then, might one optimize volunteer measures in that more complex medical environment? And in situations where volunteers will travel via a mixture of paths that are not easily classified as either $L_1$ travel or $L_2$ travel, what then might we expect in optimal solutions?

Future research within the sphere of measure optimization might develop dual bounds that can complement the fc-FW algorithm to, e.g., provide stopping rules and certificates of (near) optimality. One might also develop methods and complexity analyses for non-convex objective functions that arise in more complex applications. It is also of interest to explore optimization applications where artificial discretization has been the traditional first step, but only for computational convenience. With new, measure-optimization techniques, perhaps artificial discretization can be avoided or reduced.

%
%
%

\clearpage

\begin{APPENDICES}

\section{Proof of Theorem~\ref{thm:feasibleregion}}\label{sec:prooffeasibleregion}
The proof of the first assertion is straightforward. Let's now prove the second assertion of the theorem.  

Let $C(\mbox{cvx}(\mathcal{S}), \mathbb{R})$ be the space on continuous real-valued functions on $\mbox{cvx}(\mathcal{S})$ equipped with the sup-norm $$\|f\|_{\infty} := \sup \{|f(x)|: x \in \mbox{cvx}(\mathcal{S})\}.$$ We know that the space $C(\mbox{cvx}(\mathcal{S}), \mathbb{R})$ is \emph{separable}, that is, there exists a countable dense set $\{f_j, j \geq 1\}$ in  $C(\mbox{cvx}(\mathcal{S}), \mathbb{R})$. 

Suppose $\{\mu_n, n \geq 1\}$ is a sequence of probability measures on the compact metric space $(\mbox{cvx}(\mathcal{S}), \| \cdot\|)$. Denote $\mu(f) : = \int f \, d\mu$ for $\mu \in \mathcal{M}^+(\mbox{cvx}(\mathcal{S}),b), f \in C(\mbox{cvx}(\mathcal{S}),\mathbb{R})$ and consider the real-valued sequence $\{\mu_n(f_1), n\geq 1\}$. Since $\sup_{n} \mu_n(f_1) \leq b\|f_1\|_{\infty},$ we know from Bolzano-Weierstrass that there exists a convergent sub-sequence $\{\mu_n^{(1)}(f_1), n \geq 1\}$. Next consider the real-valued sequence $\{\mu_n^{(1)}(f_2), n \geq 1\}.$ Like before, since $\sup_{n} \mu_n(f_2) \leq b\|f_2\|_{\infty},$ we can extract a convergent sub-sequence $\{\mu_n^{(2)}(f_2), n \geq 1\}$. By repeating this procedure, we get a sequence $\{\mu_n^{(i)}, n \geq 1\}$ of nested sequences such that the real-valued sequence $\{\mu_n^{(i)}(f_j), n \geq 1\}$ is convergent for all $j \leq i$. In particular, we see that the real-valued sequence $\{\mu^{(n)}_n(f_j), n \geq 1\}$ converges for each $j \geq 1$. 

Let's now prove that for each $f \in C(\mbox{cvx}(\mathcal{S}),\mathbb{R}),$ the sequence $\{\mu^{(n)}_n(f), n \geq 1\}$ is Cauchy. Since $\{f_j, j \geq 1\}$ is dense in $C(\mbox{cvx}(\mathcal{S}),\mathbb{R}),$ given $\epsilon >0$, we can find $j(\epsilon)$ so that \begin{equation}\label{dense} \|f - f_{j(\epsilon)}\|_{\infty} \leq \epsilon. \end{equation} And, since $\{\mu^{(n)}_n(f_{j(\epsilon)}), n \geq 1\}$ converges (from prior arguments), we can find $N(\epsilon)$ so that for all $n,m \geq N(\epsilon)$, \begin{equation}\label{cauchy} |\mu_n^{(n)}(f_{j(\epsilon)}) - \mu_n^{(n)}(f_{j(\epsilon)}) | \leq \epsilon.\end{equation}
Therefore, using~\eqref{dense} and~\eqref{cauchy}, we see that for all $n,m \geq N(\epsilon)$, \begin{align}\label{cauchymeasureseq}
| \mu_n^{(n)}(f) - \mu_m^{(m)}(f) | & \leq \left| \mu_n^{(n)}(f) - \mu_n^{(n)}(f_{j(\epsilon)}) \right | + \left| \mu_n^{(n)}(f_{j(\epsilon)}) - \mu_m^{(m)}(f_{j(\epsilon)}) \right | + \left | \mu_m^{(m)}(f_{j(\epsilon)}) - \mu_m^{(m)}(f) \right | \nonumber \\
& \leq 3\epsilon,\end{align}
implying that the real-valued sequence $\{\mu_n^{(n)}(f), n \geq 1\}$ is convergent. Moreover, suppose we write \begin{equation}\label{wdef}w(f) := \frac{1}{b}\,\lim_{n \to \infty} \mu_n^{(n)}(f).\end{equation}
Then, applying the Riesz representation theorem, there exists $\mu \in \mathcal{M}^+(\mbox{cvx}(\mathcal{S}),b)$ such that \begin{equation*} \mu^{(n)}_n(f) := \int f \, d\mu^{(n)}_n \to \int f \, d\mu =: \mu(f),\end{equation*} allowing us to conclude that $\{\mu_n^{(n)}, n \geq 1\}$ weak$^*$ converges, and $\mathcal{M}^+(\mbox{cvx}(\mathcal{S}),b)$ is weak$^*$ compact. 
\section{Smoothness}
Important for optimization, we show that the objective \( J \) in the emergency response problem~\eqref{restrictedopt} is \((2b+1)\)-smooth over the domain \(\mathcal{M}^+(\mathrm{cvx}(\mathcal{S}), b)\).

\begin{lemma}[Smoothness] The functional \( J \) satisfies the Lipschitz-type condition: \begin{align}\label{smoothness} \sup_{x \in \mathrm{cvx}(\mathcal{S})} |h_{\mu_1}(x)-h_{\mu_2}(x) |  \quad \leq \quad (2b+1) \|\mu_1 - \mu_2\|.
\end{align}
\end{lemma}

\begin{proof}{Proof.} We aim to bound
\begin{align} \label{normDiffUb}
    \sup_{x \in \mathrm{cvx}(\mathcal{S})} |h_{\mu_1}(x) - h_{\mu_2}(x)| 
    &= \sup_{x \in \mathrm{cvx}(\mathcal{S})} \Bigg| \int_{\mathcal{S}} \eta(dy) \int_0^{\infty}  \Big[
    \left( \mu_1\left(\bar{B}(y,t)\right) - b\, \mathbb{I}_{[\|y - x\|, \infty)}(t) \right) \nonumber \\
    & \quad \times \left( e^{-\mu_1\left(\bar{B}(y,t)\right)} - e^{-\mu_2\left(\bar{B}(y,t)\right)} \right) 
    + \left( \mu_1\left(\bar{B}(y,t)\right) - \mu_2\left(\bar{B}(y,t)\right) \right) 
    e^{-\mu_2\left(\bar{B}(y,t)\right)} \Big] \, d\death(t) \Bigg|.
\end{align}

For fixed \( y \in \mathcal{S} \), \( t > 0 \), and without loss of generality assuming \( \mu_1\left(\bar{B}(y,t)\right) \geq \mu_2\left(\bar{B}(y,t)\right) \), we have
\begin{align}
    \left| e^{-\mu_1\left(\bar{B}(y,t)\right)} - e^{-\mu_2\left(\bar{B}(y,t)\right)} \right| 
    = e^{-\mu_1\left(\bar{B}(y,t)\right)} \left(1 - e^{-(\mu_1 - \mu_2)\left(\bar{B}(y,t)\right)} \right) \leq (\mu_1 - \mu_2)\left(\bar{B}(y,t)\right) \leq \|\mu_1 - \mu_2\|.
\end{align}

Also, since \( \left| \mu_1\left(\bar{B}(y,t)\right) - b\, \mathbb{I}_{[\|y - x\|, \infty)}(t) \right| \leq 2b \), the integrand in \eqref{normDiffUb} is bounded by:
\[
2b \cdot \left| e^{-\mu_1\left(\bar{B}(y,t)\right)} - e^{-\mu_2\left(\bar{B}(y,t)\right)} \right| + \left| \mu_1\left(\bar{B}(y,t)\right) - \mu_2\left(\bar{B}(y,t)\right) \right| \cdot e^{-\mu_2\left(\bar{B}(y,t)\right)} \leq (2b + 1) \|\mu_1 - \mu_2\|.
\]

Plugging this bound into \eqref{normDiffUb}, and using
$\int_{\mathcal{S}} \eta(dy) \int_0^\infty d\death(t) = 1,
$ we conclude that
\[
\sup_{x \in \mathrm{cvx}(\mathcal{S})} |h_{\mu_1}(x) - h_{\mu_2}(x)| \leq (2b+1) \|\mu_1 - \mu_2\|,
\]
that is, the functional \( J \) satisfies the Lipschitz condition~\eqref{smoothness}.
\hfill \Halmos
\end{proof}

As a consequence of \( L \)-smoothness, we immediately obtain the following inequality for $J$ (see~\cite{yu2024det}).

\begin{lemma}[Smooth Functional Inequality~\cite{yu2024det}] For any $\mu,\tilde{\mu} \in \mathcal{M}^+(\mathrm{cvx}(\mathcal{S}), b)$, $J$ satisfies \begin{equation}\label{descentlemmagen} 0 \leq J(\tilde{\mu}) - \left( J(\mu) + J'_{\mu}(\tilde{\mu}-\mu) \right)  \leq  \frac{L}{2b}\|\mu-\tilde{\mu}\|^2, \end{equation}
where $L=2b+1$.
\end{lemma}

\end{APPENDICES}

\clearpage

\section*{Acknowledgments.}
This work was partially supported by National Science Foundation grants CMMI-2035086, DMS-2230023 and OAC-2410950. We thank the editorial team for helpful reports that improved the content and exposition of the paper.


\bibliographystyle{informs2014} 
\bibliography{references,stochastic_optimization}

\begin{thebibliography}{33}
\providecommand{\natexlab}[1]{#1}
\providecommand{\url}[1]{\texttt{#1}}
\providecommand{\urlprefix}{URL }

\bibitem[{B{\ae}kgaard et~al.(2017)B{\ae}kgaard, Viereck, M{\o}ller, Ersb{\o}ll, Lippert, \protect\BIBand{} Folke}]{baekgaard2017effects}
B{\ae}kgaard JS, Viereck S, M{\o}ller TP, Ersb{\o}ll AK, Lippert F, Folke F (2017) The effects of public access defibrillation on survival after out-of-hospital cardiac arrest: a systematic review of observational studies. \emph{Circulation} 136(10):954--965.

\bibitem[{Boyd et~al.(2017)Boyd, Schiebinger, \protect\BIBand{} Recht}]{2017boygeorec}
Boyd N, Schiebinger G, Recht B (2017) The alternating descent conditional gradient method for sparse inverse problems. \emph{SIAM Journal on Optimization} 27(2).

\bibitem[{Bredies \protect\BIBand{} Pikkarainen(2013)}]{2013brepik}
Bredies K, Pikkarainen H (2013) Inverse problems in spaces of measures. \emph{ESAIM. Control, optimisation and calculus of variations} 19(1):190--218, ISSN 1292-8119.

\bibitem[{Bubeck et~al.(2015)}]{2015bub}
Bubeck S, et~al. (2015) Convex optimization: Algorithms and complexity. \emph{Foundations and Trends{\textregistered} in Machine Learning} 8(3-4):231--357.

\bibitem[{Chizat(2022)}]{2022bchi}
Chizat L (2022) Sparse optimization on measures with over-parameterized gradient descent. \emph{Mathematical programming} 194(1-2):487--532, ISSN 0025-5610.

\bibitem[{Chizat \protect\BIBand{} Bach(2018)}]{chizat2018global}
Chizat L, Bach F (2018) On the global convergence of gradient descent for over-parameterized models using optimal transport. \emph{Advances in neural information processing systems} 31.

\bibitem[{Chu et~al.(2019)Chu, Blanchet, \protect\BIBand{} Glynn}]{2019chublagly}
Chu C, Blanchet J, Glynn P (2019) Probability functional descent: A unifying perspective on {GAN}s, variational inference, and reinforcement learning.

\bibitem[{De~Maio et~al.(2003)De~Maio, Stiell, Wells, Spaite, Group et~al.}]{de2003optimal}
De~Maio VJ, Stiell IG, Wells GA, Spaite DW, Group OPALSS, et~al. (2003) Optimal defibrillation response intervals for maximum out-of-hospital cardiac arrest survival rates. \emph{Annals of emergency medicine} 42(2):242--250.

\bibitem[{Denoyelle et~al.(2019)Denoyelle, Duval, Peyr{\'e}, \protect\BIBand{} Soubies}]{denoyelle2019sliding}
Denoyelle Q, Duval V, Peyr{\'e} G, Soubies E (2019) The sliding {Frank--Wolfe} algorithm and its application to super-resolution microscopy. \emph{Inverse Problems} 36(1):014001.

\bibitem[{Dunn \protect\BIBand{} Harshbarger(1978)}]{1978dunhar}
Dunn JC, Harshbarger S (1978) Conditional gradient algorithms with open loop step size rules. \emph{Journal of Mathematical Analysis and Applications} 62(2):432--444.

\bibitem[{Eftekhari \protect\BIBand{} Thompson(2019)}]{eftekhari2019sparse}
Eftekhari A, Thompson A (2019) Sparse inverse problems over measures: Equivalence of the conditional gradient and exchange methods.

\bibitem[{Fernholz(1983)}]{1983fer}
Fernholz L (1983) Lecture notes in statistics. \emph{Von Mises Calculus for Statistical Functionals} 19.

\bibitem[{Fernholz(2011)}]{2011fer}
Fernholz L (2011) Functional derivatives in statistics: Asymptotics and robustness.

\bibitem[{{GoodSAM Platform}(2020)}]{goodsam2020}
{GoodSAM Platform} (2020) {GoodSAM}. \url{https://www.goodsamapp.org/}, accessed February 29, 2020.

\bibitem[{Kent et~al.(2021)Kent, Blanchet, \protect\BIBand{} Glynn}]{kent2021frankwolfe}
Kent C, Blanchet J, Glynn P (2021) {Frank-Wolfe} methods in probability space.

\bibitem[{Kiefer(1960)}]{1960kie}
Kiefer J (1960) Optimum experimental designs {V}, with applications to systematic and rotatable designs. \emph{Proceedings of the fourth Berkeley symposium on mathematical statistics and probability}, volume~1, 381--405 (Univ of California Press).

\bibitem[{Kiefer(1974)}]{1974kie}
Kiefer J (1974) General equivalence theory for optimum designs (approximate theory). \emph{The Annals of Statistics} 2(5):849--879.

\bibitem[{Kingma \protect\BIBand{} Ba(2017)}]{kingma2017adam}
Kingma DP, Ba J (2017) Adam: A method for stochastic optimization. \urlprefix\url{https://arxiv.org/abs/1412.6980}.

\bibitem[{Mei et~al.(2018)Mei, Montanari, \protect\BIBand{} Nguyen}]{2018mei}
Mei S, Montanari A, Nguyen PM (2018) A mean field view of the landscape of two-layer neural networks. \emph{Proceedings of the National Academy of Sciences} 115(33), \urlprefix\url{http://dx.doi.org/10.1073/pnas.1806579115}.

\bibitem[{Molchanov \protect\BIBand{} Zuyev(2002)}]{2002molzuy}
Molchanov I, Zuyev S (2002) Steepest descent algorithms in a space of measures. \emph{Statistics and Computing} 12:115--123.

\bibitem[{Nichol et~al.(1999)Nichol, Stiell, Laupacis, De~Maio, Wells et~al.}]{nichol1999cumulative}
Nichol G, Stiell IG, Laupacis A, De~Maio VJ, Wells GA, et~al. (1999) A cumulative meta-analysis of the effectiveness of defibrillator-capable emergency medical services for victims of out-of-hospital cardiac arrest. \emph{Annals of emergency medicine} 34(4):517--525.

\bibitem[{Nitanda \protect\BIBand{} Suzuki(2017)}]{nitanda2017stoch}
Nitanda A, Suzuki T (2017) Stochastic particle gradient descent for infinite ensembles. \emph{arXiv preprint arXiv:1712.05438} .

\bibitem[{Okabe(2000)}]{2000oka}
Okabe A (2000) \emph{Spatial tessellations: concepts and applications of Voronoi diagrams}. Wiley series in probability and statistics. Applied probability and statistics section (Chichester: Wiley), 2nd edition, ISBN 0471986356.

\bibitem[{Oving et~al.(2019)Oving, Masterson, Tjelmeland, Jonsson, Semeraro, Ringh, Truhlar, Cimpoesu, Folke, Beesems et~al.}]{oving2019first}
Oving I, Masterson S, Tjelmeland IB, Jonsson M, Semeraro F, Ringh M, Truhlar A, Cimpoesu D, Folke F, Beesems SG, et~al. (2019) First-response treatment after out-of-hospital cardiac arrest: a survey of current practices across 29 countries in {Europe}. \emph{Scandinavian Journal of Trauma, Resuscitation and Emergency Medicine} 27:1--20.

\bibitem[{{PulsePoint}(2020)}]{pulsepoint2020}
{PulsePoint} (2020) Pulsepoint. \url{https://www.pulsepoint.org/}, accessed: 2020-02-29.

\bibitem[{Resnick(1987)}]{res1987}
Resnick S (1987) \emph{Extreme Values, Regular Variation, and Point Processes} (New York, NY: Springer).

\bibitem[{Resnick(1992)}]{1992res}
Resnick S (1992) \emph{Adventures in Stochastic Processes} (Boston, MA.: Birkh\"{a}user).

\bibitem[{Sasson et~al.(2010)Sasson, Rogers, Dahl, \protect\BIBand{} Kellermann}]{sasson2010predictors}
Sasson C, Rogers MA, Dahl J, Kellermann AL (2010) Predictors of survival from out-of-hospital cardiac arrest: a systematic review and meta-analysis. \emph{Circulation: Cardiovascular Quality and Outcomes} 3(1):63--81.

\bibitem[{van~den Berg et~al.(2024)van~den Berg, Henderson, Jagtenberg, \protect\BIBand{} Li}]{van2024modeling}
van~den Berg PL, Henderson SG, Jagtenberg CJ, Li H (2024) Modeling the impact of community first responders. \emph{Management Science} .

\bibitem[{Waalewijn et~al.(2001)Waalewijn, de~Vos, Tijssen, \protect\BIBand{} Koster}]{waalewijn2001survival}
Waalewijn RA, de~Vos R, Tijssen JG, Koster RW (2001) Survival models for out-of-hospital cardiopulmonary resuscitation from the perspectives of the bystander, the first responder, and the paramedic. \emph{Resuscitation} 51(2):113--122.

\bibitem[{Yan et~al.(2020)Yan, Gan, Jiang, Wang, Chen, Luo, Zong, Chen, \protect\BIBand{} Lv}]{yan2020global}
Yan S, Gan Y, Jiang N, Wang R, Chen Y, Luo Z, Zong Q, Chen S, Lv C (2020) The global survival rate among adult out-of-hospital cardiac arrest patients who received cardiopulmonary resuscitation: a systematic review and meta-analysis. \emph{Critical care} 24:1--13.

\bibitem[{Yan et~al.(2024)Yan, Wang, \protect\BIBand{} Rigollet}]{yan2024learning}
Yan Y, Wang K, Rigollet P (2024) Learning {Gaussian} mixtures using the {Wasserstein--Fisher--Rao} gradient flow. \emph{The Annals of Statistics} 52(4):1774--1795.

\bibitem[{Yu et~al.(2024)Yu, Henderson, \protect\BIBand{} Pasupathy}]{yu2024det}
Yu D, Henderson SG, Pasupathy R (2024) Deterministic and stochastic {Frank-Wolfe} recursion on probability spaces. \urlprefix\url{https://arxiv.org/abs/2407.00307}.

\end{thebibliography}

\end{document}